\documentclass[11pt]{elsarticle}
\usepackage{fullpage}
\usepackage{multirow}
\usepackage{amsmath}
\usepackage{amssymb}
\usepackage{enumitem}
\usepackage{xfrac}
\usepackage{booktabs}
\usepackage[thmmarks]{ntheorem}
\usepackage{graphicx}
\usepackage{float}
\usepackage{caption}
\usepackage{subcaption}
\usepackage{algorithm}
\usepackage{algpseudocode}
\usepackage{hyperref}
\usepackage{xcolor}
\usepackage{mathtools}
\usepackage{bm}
\usepackage{tabularx}
\usepackage{tikz}
\usepackage{empheq}
\usepackage{adjustbox}

\usepackage{nicefrac}

\usepackage{lmodern}
\usepackage{comment}

\usepackage{siunitx}
\newcounter{problem}[section]

\theoremseparator{.}
\newtheorem{theorem}{Theorem}[section]

\newtheorem{remark}{Remark}[section]

\theorembodyfont{\normalfont}

\theoremstyle{nonumberplain}
\theoremheaderfont{\itshape}
\theorembodyfont{\normalfont}
\theoremseparator{.}
\theoremsymbol{\ensuremath{\square}}
\newtheorem{proof}{Proof}

\numberwithin{equation}{section}

\allowdisplaybreaks

\begin{document}

\begin{frontmatter}

\title{Structure-preserving generalized transferable neural networks for the Cahn-Hilliard equation}

\author[au]{Cao-Kha Doan}\corref{mycorrespondingauthor} \ead{kcd0030@auburn.edu}
\author[au]{Thi-Thao-Phuong Hoang} \ead{tzh0059@auburn.edu}
\cortext[mycorrespondingauthor]{Corresponding author}
\author[usc]{Lili Ju} \ead{ju@math.sc.edu}
\author[usc]{Shuting Wang} \ead{sw156@email.sc.edu}

\address[au]{Department of Mathematics and Statistics, Auburn University, Auburn, AL 36849, USA}
\address[usc]{Department of Mathematics, University of South Carolina, Columbia, SC 29208, USA}

\begin{abstract}
This paper is concerned with a structure-preserving neural network-based framework for the Cahn-Hilliard equation in mixed form. We employ generalized transferable neural networks (GTransNet) for spatial approximation and stabilized backward differentiation formulas (BDF) for temporal discretization. The resulting first- and second-order in time GTransNet-BDF schemes are shown to conserve mass and satisfy energy stability at the time-discrete level. The schemes are implemented by a collocation-based method, in which a least-squares system with constant coefficient matrix needs to be solved at each time step. The solution of this system, which determines the output-layer weights of the network, violates mass conservation due to the expected nonzero least-squares residual. To overcome this issue, we introduce a novel post-processing mass-conserving projection that enforces the mass constraint through a minimization problem, whose solution can be computed at negligible computational cost. A key advantage of the proposed method lies in its predetermined hidden layers and mesh-free nature, making the method applicable to complex domains, variable mobility, and long-time simulations. Extensive numerical experiments in two and three dimensions verify convergence, mass conservation, and energy dissipation as well as demonstrate the accuracy and robustness of the proposed GTransNet-BDF schemes.

\end{abstract}


\end{frontmatter}


\section{Introduction}

The Cahn-Hilliard (CH) equation~\cite{Cahn58} is a fundamental model in materials science and applied mathematics that describes the process of phase separation in binary alloys, a phenomenon known as spinodal decomposition. It also arises in a wide range of applications, including tumor growth modeling~\cite{Cristini10}, image processing~\cite{Bertozzi07}, thin film dynamics~\cite{Thiele01}, and two-phase flow simulations~\cite{Doan25b}. As a fourth-order nonlinear parabolic PDE, the CH equation presents significant computational challenges: its solutions develop thin interfacial layers controlled by a small parameter, and the dynamics involve multiple time scales, from fast initial spinodal decomposition to slow late-stage coarsening driven by the Ostwald ripening mechanism. Classically, the CH equation has been solved using a variety of mesh-based numerical methods, including finite difference~\cite{Furihata01,Ju15}, finite element~\cite{Elliott87,DuNicolaides91}, spectral and pseudo-spectral~\cite{Chen98,Zhu99,Cheng16}, and finite volume~\cite{Cueto08} methods. Due to the sharp interfacial structures, these methods typically require very fine meshes near the interface to accurately resolve the steep gradients, and therefore often rely on adaptive mesh refinement (AMR) strategies~\cite{Ceniceros07,Wise08}. While AMR significantly reduces the computational cost compared to uniform fine meshes, it introduces substantial implementation complexity, particularly for problems in three dimensions or on complex domains, and requires sophisticated mesh tracking and re-meshing as interfaces evolve over time.

For the time integration of the CH equation, semi-implicit schemes based on backward differentiation formulas (BDF) have become a standard approach~\cite{Shen10,Wang18,Li16}. The key idea is to treat the stiff linear fourth-order term implicitly for stability while evaluating the nonlinear term explicitly to avoid solving nonlinear systems at each time step. A stabilization technique is commonly employed, in which a linear term is added to the implicit part and subtracted from the explicit nonlinear part, yielding energy-stable schemes with appropriate choices of the stabilization constant~\cite{Xu06,Wang18}. In addition, the scalar auxiliary variable (SAV)~\cite{Shen18SAV} and Lagrange multiplier~\cite{Cheng20,Hou23} approaches, along with their variants, have provided a systematic framework for constructing unconditionally energy-stable schemes for gradient flow models.

In recent years, deep learning has emerged as a powerful tool for scientific computing, and neural network-based PDE solvers have attracted intense research interest. Physics-informed neural networks (PINNs)~\cite{Raissi19} embed the governing equations and boundary conditions directly into the loss function and train the network via stochastic gradient descent (SGD). While PINNs offer the attractive feature of being mesh-free, their application to the CH equation has proven particularly challenging. It was demonstrated in~\cite{Wight21} that a direct application of standard PINNs to the Allen-Cahn and CH equations does not provide accurate solutions in many cases, primarily due to sharp transition layers that evolve over time and the inability of fixed collocation points to adapt to these moving features. Several remedies have been proposed, including sequential time-marching strategies~\cite{Mattey22}, adaptive collocation point resampling~\cite{Wight21}, neural tangent kernel (NTK) based loss weighting~\cite{Chen25PFPINN}, and mass-preserving constraints~\cite{Huang24}. However, these approaches often rely on iterative SGD-based optimization of deep neural networks, which is inherently expensive and can suffer from convergence difficulties, local minima, and hyperparameter sensitivity.

An alternative paradigm that has gained significant attention is the class of shallow (single-hidden-layer) neural network methods with predetermined hidden-layer parameters. Representative methods in this class include the Extreme Learning Machine (ELM)~\cite{Huang06,Dong21,DongLi21}, the Random Feature Method (RFM)~\cite{Chen22RFM}, and the Transferable Neural Network (TransNet)~\cite{Zhang24}. In these approaches, the weights and biases of the hidden-layer neurons are fixed in advance -- either by random sampling from prescribed distributions or by deliberate geometric construction -- and only the weights of the output layer are optimized, typically by solving a linear least-squares problem. This strategy completely eliminates the need for SGD training, resulting in orders-of-magnitude speedups over deep-network-based solvers. Among these methods, TransNet stands out by employing a geometrically interpretable construction that produces uniformly distributed partition hyperplanes in the hidden layer, ensuring balanced expressive power across the computational domain~\cite{Zhang24}. The method has been successfully applied to various steady-state~\cite{Zhang24,Cheng26} and time-dependent PDE problems~\cite{ZhangBDF24}, as well as interface problems via domain decomposition~\cite{Lu25} and singularly perturbed problems via matched asymptotic expansions~\cite{Shen26MAE}.

Recently, the Generalized Transferable Neural Network (GTransNet) was proposed in~\cite{Cheng26} to address the limitation of TransNet in handling problems with highly oscillatory solutions. GTransNet augments the original TransNet with additional hidden layers while preserving its predetermined feature-generation mechanism. The first hidden layer retains TransNet's uniform partition hyperplane construction with a symmetric bias distribution, while the subsequent hidden layers employ a variance-controlled weight sampling strategy that prevents neuron activations from saturating, thereby significantly enhancing the network's expressive capacity for capturing high-frequency and multiscale features. The GTransNet framework has demonstrated superior performance over TransNet for a broad class of steady-state PDEs, including the Poisson, Helmholtz, multiscale elliptic, Navier-Stokes, and Allen-Cahn equations~\cite{Cheng26}. However, its application to time-dependent fourth-order problems remains unexplored.

In this paper, we develop the so-called GTransNet-BDF method for solving the CH equation. To avoid numerical sensitivity associated with the biharmonic operator, we adopt a mixed formulation that reduces the CH equation to a coupled system of two second-order equations and keeps the derivative computations manageable. The proposed framework employs the GTransNet basis for spatial approximation and stabilized BDF method for time integration, resulting in the first- and second-order in time GTransNet-BDF schemes. In contrast to many existing TransNet-based approaches that treat the time variable as an extra spatial dimension, this framework provides more flexibility in time discretization and is well suited to long-time simulations. To the best of our knowledge, this is the first application of transferable neural networks to a fourth-order PDE. The main contributions and features of our work are summarized as follows:

\begin{itemize}[leftmargin=*]
\item \textbf{Mesh-free solver for the CH equation with predetermined network parameters.} The GTransNet hidden-layer parameters are fixed a priori based on geometric principles (uniformly distributed partition hyperplanes and variance-controlled weight sampling), and only the output-layer weights are computed by solving a linear least-squares system at each time step. This eliminates the need for mesh generation, adaptive mesh refinement, interface tracking, and iterative SGD optimization, resulting in a simple and efficient computational framework.

\item \textbf{Mass conservation and energy stability.}
Two intrinsic properties of the CH equation, namely mass conservation
and energy dissipation, are not automatically satisfied by existing neural-network solvers. Our framework preserves both properties at the discrete level. Energy stability is achieved through the standard stabilization technique combined with semi-implicit BDF time integration. Discrete mass conservation, which the over-determined least-squares system satisfies only approximately, is enforced exactly through a projection step that corrects the least-squares solution along a single precomputed direction, with minimal perturbation to the PDE residual and negligible additional cost.

\item \textbf{Efficient time marching with a one-time QR factorization.} Since the coefficient matrix of the least-squares system is time-independent, a single QR factorization computed at the start of the simulation suffices for the entire time-stepping process, and each subsequent time step is solved efficiently via a matrix-vector product followed by back substitution.

\item \textbf{Adaptability to complex domain geometries.} The mesh-free nature of the GTransNet-BDF method, combined with the domain-covering ball construction for generating hidden-layer neurons, allows the method to be straightforwardly applied to problems on irregular or complex domains without the need for specialized mesh generators.
\end{itemize}

The remainder of this paper is organized as follows. In Section~\ref{sec:GTransNet}, we provide a brief review of TransNet for solving PDEs and present the generalized version (GTransNet) along with its theoretical properties. The combination of GTransNet and BDF time-stepping method for the CH equation is proposed in Section~\ref{sec:GTransNet_BDF}, together with its efficient implementation based on least-squares formulation. Section~\ref{sec:properties} establishes mass conservation and energy stability of the proposed GTransNet-BDF schemes at the space-continuous level, while a mass-conserving projection that accounts for the least-squares residual is described in Section~\ref{sec:projection}. Numerical experiments on convergence test, shape relaxation, and coarsening dynamics in two and three dimensions are reported in Section~\ref{sec:num}. Finally, some concluding remarks are given in Section~\ref{sec:conclusion}.

\section{Generalized transferable neural networks}\label{sec:GTransNet}
\subsection{TransNet}
For simplicity of presentation, we describe TransNet on the unit ball $B_1(\bm 0)$; the general case follows by an affine shift (see Remark~\ref{rem:affine_shift} below). The idea of transferable neural networks~\cite{Zhang24} is to fix the hidden layer parameters in advance, without relying on any PDE-specific information, and to optimize only the output weights. Using the $\tanh$ activation function, the approximate solution $u_{\mathrm{NN}}$ to a time-dependent PDE takes the following form:
\begin{align}\label{uNN}
\begin{aligned}
u_{\mathrm{NN}}(\bm x,t)&=\sum_{m=1}^N\alpha_m(t)\,\tanh(\bm w_m^T\bm x+b_m)=\sum_{m=1}^N\alpha_m(t)\,\tanh(\gamma_m(\bm a_m^T\bm x+r_m)),
\end{aligned}
\end{align}
where $N$ denotes the number of hidden neurons, $\bm w_m$ and $b_m$ are the weight and bias of the $m$-th hidden neuron, respectively, and $\alpha_1,\alpha_2,\ldots,\alpha_N$ are the time-dependent weights of the output layer. Note that the decomposition of $(\bm w_m,b_m)$ into the location parameter $(\bm a_m,r_m)$ with $\|\bm a_m\|_2=1$ and the shape parameter $\gamma_m>0$ in~\eqref{uNN} is based on the following relations:
\begin{align*}
\gamma_m=\|\bm w_m\|_2,\quad \bm a_m=\frac{\bm w_m}{\gamma_m},\quad r_m=\frac{b_m}{\gamma_m},\quad 1\le m\le N.
\end{align*}
Geometrically, $\bm a_m$ represents the normal direction of the \textit{partition hyperplane}, defined by $\bm a_m^T\bm x+r_m=0$, while $r_m$ denotes its distance from the origin. The shape parameter $\gamma_m$ determines the steepness of the pre-activation value $\bm w_m^T\bm x+b_m$ along the normal direction $\bm a_m$. We refer to~\cite{Zhang24} for a geometric visualization of these parameters.

To measure the density of hidden neurons within a specific region, we recall the partition hyperplane density function $D_N^{\tau}(\bm x)$, defined as
\begin{align*}
D_N^{\tau}(\bm x)=\frac{1}{N}\sum_{m=1}^N\chi_{\{d_m(\bm x)<\tau\}}(\bm x),\quad\text{where }\;d_m(\bm x)=|\bm a_m^T\bm x+r_m|.
\end{align*}
Here $d_m(\bm x)$ measures the distance from $\bm x$ to the $m$-th partition hyperplane, and $\chi$ denotes the standard indicator function. The following theorem (cf.~\cite{Zhang24,Lu25}) shows that uniform sampling of the location parameters distributes the partition hyperplanes evenly throughout the domain, ensuring that the basis functions provide balanced resolution everywhere.
\begin{theorem}\label{thm:E}
If $\{\bm a_m\}_{m=1}^N$ are i.i.d. and uniformly distributed on the $d$-dimensional unit sphere, and $\{r_m\}_{m=1}^N$ are i.i.d. and uniformly distributed in $[0,1]$, then for any $\tau\in (0,1)$,
\begin{align*}
    \mathbb E\left[D_N^{\tau}(\bm x)\right]=\tau,
\end{align*}
for all $\bm x\in \mathbb R^d$ such that $\|\bm x\|_2\le 1-\tau$.
\end{theorem}

Once the location parameters $\{(\bm a_m,r_m)\}_{m=1}^N$ are determined as in Theorem~\ref{thm:E}, we use Gaussian random fields~\cite{Zhang24} to tune the shape parameters $\{\gamma_m\}_{m=1}^N$, which are assumed to be uniform (i.e., $\gamma=\gamma_1=\ldots=\gamma_N$). The \textit{neural feature space} is then defined by
\begin{align*}
\begin{aligned}
\mathcal P_{\mathrm{NN}}&=\mathrm{span}\left\{\tanh(\gamma(\bm a_1^T\bm x+r_1)),\ldots,\tanh(\gamma(\bm a_N^T\bm x+r_N))\right\}\\
&=:\mathrm{span}\left\{\psi_1(\bm x),\dots,\psi_N(\bm x)\right\},
\end{aligned}
\end{align*}
where $\psi_m(\bm x)=\tanh(\gamma(\bm a_m^T\bm x+r_m))$ for $1\le m\le N$.

\begin{remark}\label{rem:affine_shift}
For a general domain $\Omega\subset B_R(\bm x_c)$, the $m$-th hidden neuron is defined as
$$\psi_m(\bm x)=\tanh(\gamma(\bm a_m^T(\bm x-\bm x_c)+Rr_m)),\quad 1\le m\le N,$$
where $\{\bm a_m\}_{m=1}^N$ and $\{r_m\}_{m=1}^N$ are sampled as in Theorem~\ref{thm:E}. Equivalently, the TransNet approximation can be written compactly as
\begin{align*}
\begin{cases}
\bm\psi(\bm x)=\tanh(\gamma(\bm{\mathcal A}(\bm x-\bm x_c)+R\bm r)),\\
u_{\mathrm{NN}}(\bm x,t)=\bm\alpha(t)^T\bm\psi(\bm x),
\end{cases}
\end{align*}
where $\bm{\mathcal A}=[\bm a_1,\bm a_2,\ldots,\bm a_N]^T\in\mathbb R^{N\times d}$, $\bm r=(r_1,r_2,\ldots,r_N)^T\in\mathbb R^N$, $\bm\psi=(\psi_1,\psi_2,\ldots,\psi_{N})^T\in\mathbb R^{N}$, and $\bm\alpha=(\alpha_1,\alpha_2,\ldots,\alpha_N)^T\in\mathbb R^N$.
\end{remark}

\subsection{Generalized TransNet (GTransNet)}

Although the original TransNet performs well for problems with relatively smooth solutions, its accuracy may degrade when dealing with problems exhibiting sharp interfaces, steep gradients, or high-frequency features. As discussed in~\cite{Cheng26}, when the shape parameter $\gamma$ increases, the activation values of the single hidden layer tend to cluster in the saturation regions of the $\tanh$ function, thereby limiting the network's expressive capacity. To address this limitation, we adopt the generalized TransNet (GTransNet) framework~\cite{Cheng26} to augment the original TransNet with additional hidden layers while preserving the interpretable feature-generation mechanism.

\subsubsection{Network architecture}
Let $N_l$ denote the number of neurons in the $l$-th hidden layer. The GTransNet with $L\ge 2$ hidden layers takes the following form
\begin{align}\label{GTransNet}
\begin{cases}
\bm\psi_1(\bm x)=\tanh(\gamma(\bm{\mathcal A}(\bm x-\bm x_c)+R\bm r)),\\
\bm\psi_l(\bm x)=\tanh(\bm W_l\bm\psi_{l-1}(\bm x)),\quad l=2,\ldots,L,\\
u_{\mathrm{NN}}(\bm x,t)=\bm\alpha(t)^T\bm\psi_L(\bm x),
\end{cases}
\end{align}
where $\bm{\mathcal A}$ and $\bm r$ are defined as in Remark~\ref{rem:affine_shift} with $N$ replaced by $N_1$, $\bm\psi_l=(\psi_1^{(l)},\ldots,\psi_{N_l}^{(l)})^T\in\mathbb R^{N_l}$ for $1\le l\le L$, $\bm W_l=(W_{ij}^{(l)})\in\mathbb R^{N_l\times N_{l-1}}$ for $2\le l\le L$, and $\bm\alpha=(\alpha_1,\ldots,\alpha_{N_L})^T\in\mathbb R^{N_L}$ collects the time-dependent output-layer weights. Note that no bias terms are used in the second and subsequent hidden layers.

The first hidden layer of GTransNet follows the same sampling strategy as in TransNet for the direction vectors $\{\bm a_m\}_{m=1}^{N_1}$ (cf.~Theorem~\ref{thm:E}), but the neuron biases are drawn from a centrally symmetric distribution, i.e.,
\begin{align}\label{rm_sym}
r_m\sim\mathcal U[-1,1],\quad m=1,\ldots,N_1,
\end{align}
instead of $r_m\sim\mathcal U[0,1]$ in the original TransNet. This symmetric choice makes the first-layer activations zero-mean (cf.~Theorem~\ref{lem:zero_mean}), which the one-sided distribution of TransNet does not, and this property propagates through the deeper hidden layers (cf.~Theorem~\ref{thm:zero_mean}). For the subsequent hidden layers, the weight matrices $\{\bm W_l\}_{l=2}^L$ are predetermined through a variance-controlled sampling strategy inspired by the Xavier initialization~\cite{Glorot10}:
\begin{align*}
W_{ij}^{(l)}\sim\mathcal N\left(0,\sigma_l^2\right)\quad\text{with}\quad \sigma_l=\sqrt{\frac{\delta}{N_{l-1}}},
\end{align*}
where $0<\delta\le 1$ is a variance control parameter.

\subsubsection{Theoretical properties}

\begin{theorem}[Zero mean of the first hidden layer]\label{lem:zero_mean}
For any fixed $\bm x\in\mathbb R^d$, the random variable $\psi_m^{(1)}(\bm x)$ defined by~\eqref{GTransNet} with the sampling~\eqref{rm_sym} satisfies
$$\psi_m^{(1)}(\bm x)\overset{d}{=}-\psi_m^{(1)}(\bm x),\quad m=1,\ldots,N_1,$$
where $\overset{d}{=}$ denotes equality in distribution. Consequently, $\mathbb E[\bm\psi_1(\bm x)]=\bm 0$.
\end{theorem}

\begin{theorem}[Zero mean propagation]\label{thm:zero_mean}
The means of the neuron pre-activations and activations in the subsequent hidden layers of GTransNet~\eqref{GTransNet} are all zeros, i.e.,
$$\mathbb E[\bm W_l\bm\psi_{l-1}]=\bm 0\quad\text{and}\quad \mathbb E[\bm\psi_l]=\bm 0,\quad l=2,\ldots,L.$$
\end{theorem}

\begin{theorem}[Controlled variance propagation]\label{thm:var}
Assume that each component of $\bm\psi_1$ has a variance at most $\sigma_0^2$. Then the variances of the neuron activations in the subsequent hidden layers of GTransNet~\eqref{GTransNet} satisfy
$$\mathrm{Var}[\psi_i^{(l)}]\le \delta^{l-1}\sigma_0^2,\quad i=1,\ldots,N_l,\quad l=2,\ldots,L.$$
\end{theorem}

Theorem~\ref{thm:var} shows that by properly choosing the parameter $0<\delta\le 1$, the variance of neuron values in the deeper hidden layers can be effectively reduced, preventing them from clustering in the saturation regions. This mechanism is particularly beneficial for capturing sharp interfaces and steep gradients that arise in the solutions of the CH equation.

\section{GTransNet-BDF schemes for the Cahn-Hilliard equation}\label{sec:GTransNet_BDF}

\subsection{Model problem}

For a bounded domain $\Omega\subset\mathbb R^d$ ($d\le 3$) and a terminal time $T>0$, we consider the following CH equation:
\begin{align}\label{eq:CH}
\frac{\partial u}{\partial t}=D\Delta(-\varepsilon^2\Delta u+f(u)),\quad (\bm x,t)\in \Omega\times (0,T],
\end{align}
subject to the initial condition $u(\cdot,0)=u_0$ and periodic or homogeneous Neumann boundary conditions. In~\eqref{eq:CH}, $u(\bm x,t)$ denotes the order parameter defined as the concentration difference between the two phases, $\varepsilon$ represents the interfacial thickness, $D$ is the diffusion coefficient, and $f(u)=F'(u)$, where $F(u)$ is a double-well potential. In the mixed formulation, the fourth-order CH equation~\eqref{eq:CH} is split into a system of two second-order equations:
\begin{align}\label{eq:CHmix}
    \begin{cases}
        \dfrac{\partial u}{\partial t} = D\Delta \mu, & \text{in } \Omega \times (0,T], \\
        \mu = -\varepsilon^2 \Delta u + f(u), & \text{in } \Omega \times (0,T],
    \end{cases}
\end{align}
where $\mu$ denotes the chemical potential. Note that~\eqref{eq:CH}, or equivalently~\eqref{eq:CHmix}, can be viewed as an $H^{-1}$ gradient flow with respect to the Ginzburg-Landau free energy functional:
\begin{align}
    E(u) = \int_{\Omega} \left( \frac{\varepsilon^2}{2} |\nabla u|^2 + F(u) \right) d\bm x.
    \label{eq:energy_functional}
\end{align}
Under the prescribed boundary conditions, the CH equation possesses two intrinsic properties, namely mass conservation and energy dissipation:
\begin{align}\label{eq:CH_properties}
\frac{d}{dt}\int_\Omega u(\bm x,t)\,d\bm x=0,\qquad \frac{d}{dt}E(u)\le 0,\quad \forall\,t\in (0,T].
\end{align}
Next, we construct the GTransNet-BDF schemes for the mixed system~\eqref{eq:CHmix} and establish discrete analogues of both properties~\eqref{eq:CH_properties} in Section~\ref{sec:properties}.

\subsection{GTransNet-BDF schemes}

We now develop time-stepping schemes that combine the GTransNet basis from Section~\ref{sec:GTransNet} for spatial approximation with first- and second-order BDF for time integration. Consider a uniform partition of the time interval $0=t_0<t_1<\dots<t_K=T$ with the time step size $\Delta t=\nicefrac TK$. Let $\kappa\ge 0$ be a stabilization constant and $f_{\kappa}(u)=f(u)-\kappa u$, then~\eqref{eq:CHmix} can be written equivalently as
\begin{align}\label{eq:CHmix_kappa}
    \begin{cases}
        \dfrac{\partial u}{\partial t} = D\Delta \mu, & \text{in } \Omega \times (0,T], \\
        \mu = -\varepsilon^2 \Delta u+\kappa u + f_{\kappa}(u), & \text{in } \Omega \times (0,T].
    \end{cases}
\end{align}
In what follows, let $\{\phi_j(\bm x)\}_{j=1}^{N_L}$ denote the components of the last hidden layer $\bm\psi_L$ of the GTransNet \eqref{GTransNet}; that is, $\phi_j:=\psi_j^{(L)}$ for $j=1,\ldots,N_L$. The approximate solutions $u_{\mathrm{NN}}$ and $\mu_{\mathrm{NN}}$ for the order parameter $u$ and the chemical potential $\mu$ are then expressed as
\begin{align}\label{u_mu_GTransNet}
u(\bm x,t)\approx u_{\mathrm{NN}}(\bm x,t)=\sum_{j=1}^{N_L}\alpha_j(t)\,\phi_j(\bm x),\quad \mu(\bm x,t)\approx \mu_{\mathrm{NN}}(\bm x,t)=\sum_{j=1}^{N_L}\beta_j(t)\,\phi_j(\bm x).
\end{align}
Substituting the GTransNet approximations~\eqref{u_mu_GTransNet} into the stabilized mixed system~\eqref{eq:CHmix_kappa} and discretizing in time by the BDF method with explicit treatment of the nonlinear term, we obtain the first-order GTransNet-BDF1 scheme
\begin{subequations}\label{GBDF1}
\begin{align}
\sum_{j=1}^{N_L}\frac{\alpha_j^{n+1}-\alpha_j^{n}}{\Delta t}\phi_j(\bm x)&=D\sum_{j=1}^{N_L}\beta_j^{n+1}\Delta\phi_j(\bm x),\label{GBDF1a}\\
\sum_{j=1}^{N_L}\beta_j^{n+1}\phi_j(\bm x)&=\sum_{j=1}^{N_L}\alpha_j^{n+1}(-\varepsilon^2\Delta+\kappa)\phi_j(\bm x)+f_{\kappa}(u_{\mathrm{NN}}^n(\bm x)),\label{GBDF1b}
\end{align}
\end{subequations}
and the second-order GTransNet-BDF2 scheme
\begin{subequations}\label{GBDF2}
\begin{align}
\sum_{j=1}^{N_L}\frac{3\alpha_j^{n+1}-4\alpha_j^n+\alpha_j^{n-1}}{2\Delta t}\phi_j(\bm x)&=D\sum_{j=1}^{N_L}\beta_j^{n+1}\Delta\phi_j(\bm x),\label{GBDF2a}\\
\sum_{j=1}^{N_L}\beta_j^{n+1}\phi_j(\bm x)&=\sum_{j=1}^{N_L}\alpha_j^{n+1}(-\varepsilon^2\Delta+\kappa)\phi_j(\bm x)+f_{\kappa}(2u_{\mathrm{NN}}^n(\bm x)-u_{\mathrm{NN}}^{n-1}(\bm x)),\label{GBDF2b}
\end{align}
\end{subequations}
where $u_{\mathrm{NN}}^k(\bm x)=\sum_{j=1}^{N_L}\alpha_j^k\,\phi_j(\bm x)$ for $k\in\{n-1,n\}$. The schemes are initialized by fitting the initial condition $u_0$ in the least-squares sense to obtain $\{\alpha_j^0\}_{j=1}^{N_L}$, and for the GTransNet-BDF2 scheme, the solution $\{\alpha_j^{1}\}_{j=1}^{N_L}$ at $t=t_1$ is computed using the GTransNet-BDF1 scheme~\eqref{GBDF1}. We note that the nonlinear term in~\eqref{GBDF2b} is treated as $f_\kappa(2u_{\mathrm{NN}}^n-u_{\mathrm{NN}}^{n-1})$, rather than by the standard extrapolation $2f_\kappa(u_{\mathrm{NN}}^n)-f_\kappa(u_{\mathrm{NN}}^{n-1})$, to relax the time step size restriction in the energy stability analysis (cf.~Remark~\ref{rem:nonlinear_egy}).

For homogeneous Neumann boundary conditions, we impose
\begin{align}\label{BC_Neumann}
\sum_{j=1}^{N_L}\alpha_j^{n}\nabla\phi_j(\bm x)\cdot\bm n=\sum_{j=1}^{N_L}\beta_j^{n}\nabla\phi_j(\bm x)\cdot\bm n=0,\quad \forall\bm x\in \partial\Omega,\; 1\le n\le K,
\end{align}
where $\bm n$ denotes the outward unit normal vector to $\partial\Omega$. For periodic boundary conditions, we take $\Omega=(0,1)^d$ for the purpose of presentation and impose
\begin{align}\label{BC_periodic}
\begin{aligned}
\sum_{j=1}^{N_L}\alpha_j^{n}
  \left(\partial_{x_i}^k\phi_j\big|_{x_i=0}
       -\partial_{x_i}^k\phi_j\big|_{x_i=1}\right) &= 0,\quad 1\le n\le K, \\
\sum_{j=1}^{N_L}\beta_j^{n}
  \left(\partial_{x_i}^k\phi_j\big|_{x_i=0}
       -\partial_{x_i}^k\phi_j\big|_{x_i=1}\right) &= 0,\quad 1\le n\le K,
\end{aligned}
\end{align}
for each coordinate direction $i\in\{1,2,\ldots,d\}$ and $k\in\{0,1\}$, here $\bm x=(x_1,x_2,\ldots,x_d)^T\in\partial\Omega$.

\subsection{Implementation}\label{subsec:LSQ}

Next, we describe how the GTransNet-BDF schemes~\eqref{GBDF1}--\eqref{GBDF2} are implemented in practice. Let $\{\bm x_i\}_{i=1}^{K_{\mathrm{int}}}$ and $\{\bm x_i^{\mathrm{bd}}\}_{i=1}^{K_{\mathrm{bd}}}$ denote the sets of interior and boundary collocation points, respectively. We then define the interior basis matrices and the boundary basis matrix as
\begin{align*}
\bm\Phi&=\bigl(\phi_j(\bm x_i)\bigr)\in\mathbb R^{K_{\mathrm{int}}\times N_L},\quad \bm\Phi_{\Delta}=\bigl(\Delta\phi_j(\bm x_i)\bigr)\in\mathbb R^{K_{\mathrm{int}}\times N_L},\\
\bm\Phi_{\mathrm{bd}}&=\bigl(\partial_{\bm n}\phi_j(\bm x_i^{\mathrm{bd}})\bigr)\in\mathbb R^{K_{\mathrm{bd}}\times N_L},\text{ where }\partial_{\bm n}\phi_j:=\nabla\phi_j\cdot\bm n,
\end{align*}
in the homogeneous Neumann case (and analogously, $\bm\Phi_{\mathrm{bd}}$ collects the value and first-derivative differences across paired faces in the periodic case). With $\bm\alpha^n=(\alpha_1^n,\ldots,\alpha_{N_L}^n)^T$ and $\bm\beta^n=(\beta_1^n,\ldots,\beta_{N_L}^n)^T$ the coefficient vectors of $u_{\mathrm{NN}}^n$ and $\mu_{\mathrm{NN}}^n$, respectively, the unknown vector $\bm c^{n+1}=((\bm\alpha^{n+1})^T,(\bm\beta^{n+1})^T)^T\in\mathbb R^{2N_L}$ is obtained for both schemes by solving the linear least-squares system
\begin{align}\label{LSQ}
\bm A\,\bm c^{n+1}=\bm b^{n+1}.
\end{align}
For the GTransNet-BDF1 scheme~\eqref{GBDF1}, the system matrix and right-hand side are
\begin{align*}
\bm A=\bm A_{\mathrm{BDF1}}:=\begin{pmatrix} \bm\Phi & -D\Delta t\,\bm\Phi_\Delta\\ \varepsilon^2\bm\Phi_\Delta-\kappa\bm\Phi & \bm\Phi \\ \bm\Phi_{\mathrm{bd}} & \bm 0\\ \bm 0 & \bm\Phi_{\mathrm{bd}}\end{pmatrix},\qquad \bm b^{n+1}=\bm b_{\mathrm{BDF1}}^{n+1}:=\begin{pmatrix} \bm\Phi\bm\alpha^n\\ f_\kappa(\bm\Phi\bm\alpha^n)\\ \bm 0\\ \bm 0\end{pmatrix}.
\end{align*}
For the GTransNet-BDF2 scheme~\eqref{GBDF2},
\begin{align*}
\bm A=\bm A_{\mathrm{BDF2}}:=\begin{pmatrix} \tfrac32\bm\Phi & -D\Delta t\,\bm\Phi_\Delta\\ \varepsilon^2\bm\Phi_\Delta-\kappa\bm\Phi & \bm\Phi \\ \bm\Phi_{\mathrm{bd}} & \bm 0\\ \bm 0 & \bm\Phi_{\mathrm{bd}}\end{pmatrix},\qquad \bm b^{n+1}=\bm b_{\mathrm{BDF2}}^{n+1}:=\begin{pmatrix} 2\bm\Phi\bm\alpha^n-\tfrac12\bm\Phi\bm\alpha^{n-1}\\ f_\kappa(2\bm\Phi\bm\alpha^n-\bm\Phi\bm\alpha^{n-1})\\ \bm 0\\ \bm 0\end{pmatrix}.
\end{align*}
In both cases, the third and fourth block rows of $\bm A$ impose the boundary conditions on $\bm\alpha^{n+1}$ and $\bm\beta^{n+1}$, respectively. Since the coefficient matrix $\bm A$ in~\eqref{LSQ} (that is, $\bm A_{\mathrm{BDF1}}$ or $\bm A_{\mathrm{BDF2}}$) is independent of $n$, we precompute its reduced QR factorization $\bm A=\bm Q\bm R$ once and then solve
\begin{align}\label{eq:normal}
\bm R\bm c^{n+1}=\bm Q^T\bm b^{n+1}
\end{align}
via back substitution at each time step. Consequently, the per-step cost is dominated by the matrix-vector product $\bm Q^T\bm b^{n+1}$, which is more efficient than directly solving the least-squares system~\eqref{LSQ} at each time step.

At $t=0$, the initial condition $u_0(\bm x)$ is projected onto the GTransNet basis by solving the least-squares problem
\begin{align*}
\bm\alpha^0=\arg\min_{\bm\alpha\in\mathbb R^{N_L}}\left\|\bm\Phi_{\mathrm{init}}\bm\alpha-\bm u_0\right\|_2^2,
\end{align*}
where $\bm\Phi_{\mathrm{init}}=(\phi_j(\bm x_i^{\mathrm{init}}))_{i,j}$ is the basis matrix evaluated at a set of initialization points $\{\bm x_i^{\mathrm{init}}\}$, which includes both interior and boundary points, and $\bm u_0=(u_0(\bm x_i^{\mathrm{init}}))_i$.

\section{Properties of GTransNet-BDF schemes}\label{sec:properties}
In this section, we establish the mass conservation and energy stability of the first- and second-order GTransNet-BDF schemes~\eqref{GBDF1}--\eqref{GBDF2}. To that end, we define the time-discrete mass and free energy (cf.~\eqref{eq:energy_functional}) as
\begin{align*}
\mathcal M(u_{\mathrm{NN}}^n)=\int_\Omega u_{\mathrm{NN}}^n(\bm x)\,d\bm x,\quad E(u_{\mathrm{NN}}^n)=\int_\Omega\left(\frac{\varepsilon^2}{2}|\nabla u_{\mathrm{NN}}^n|^2+F(u_{\mathrm{NN}}^n)\right)d\bm x.
\end{align*}
Throughout this section, $(\cdot,\cdot)$ and $\|\cdot\|$ denote the standard $L^2(\Omega)$ inner product and associated norm, respectively. We shall restrict our attention to potential functions $F(u)$ whose derivative $f(u)=F'(u)$ satisfies the following condition
\begin{align}\label{f_lipschitz}
L_f:=\max_{u\in\mathbb R}|f'(u)|<\infty,
\end{align}
which holds under a truncation of $F$ (and hence $f$) outside a bounded interval; we refer to~\cite{Shen10,FuYang22} for the explicit form of a truncated double-well potential with quadratic growth at infinity.
For the analysis of the GTransNet-BDF2 scheme, we further recall the inverse Laplacian operator $(-\Delta)^{-1}$ defined on the mean-zero subspace
$$L_0^2(\Omega):=\Bigl\{v\in L^2(\Omega):\int_\Omega v\,d\bm x=0\Bigr\}.$$
For $v\in L_0^2(\Omega)$, we define $(-\Delta)^{-1}v=\varphi$ with $\varphi\in H^1(\Omega)\cap L_0^2(\Omega)$ solving the Poisson's equation $-\Delta\varphi=v$ under either periodic or homogeneous Neumann boundary conditions. The associated $H^{-1}$-norm is given by $\|v\|_{-1}:=\sqrt{(v,(-\Delta)^{-1}v)}$. One easily verifies the interpolation inequality
\begin{align*}
\|v\|^2\le \|\nabla v\|\cdot\|v\|_{-1},\quad \forall\,v\in H^1(\Omega)\cap L_0^2(\Omega).
\end{align*}
As a consequence, applying Young's inequality yields
\begin{align}\label{interpolation_young}
a\|v\|_{-1}^2+b\|\nabla v\|^2\ge 2\sqrt{ab}\,\|v\|^2,\quad \forall\,a,b\ge 0,\;\forall\,v\in H^1(\Omega)\cap L_0^2(\Omega).
\end{align}
Finally, we denote $\delta_tv^{n+1}:=v^{n+1}-v^n$ and $\delta_{tt}v^{n+1}:=v^{n+1}-2v^n+v^{n-1}$ for any sequence $\{v^n\}$.

By integrating the first equations~\eqref{GBDF1a} and~\eqref{GBDF2a} of the GTransNet-BDF1 and GTransNet-BDF2 schemes over $\Omega$, applying the divergence theorem, and invoking the boundary conditions~\eqref{BC_Neumann} or~\eqref{BC_periodic} on $\mu_{\mathrm{NN}}^{n+1}$, we obtain the following discrete mass conservation.

\begin{theorem}[Mass conservation]\label{thm:mass}
The first- and second-order GTransNet-BDF schemes (cf. \eqref{GBDF1} and~\eqref{GBDF2}) conserve mass unconditionally, i.e., for any time step size $\Delta t>0$, we have
\begin{align*}
\mathcal M(u_{\mathrm{NN}}^{n+1})=\mathcal M(u_{\mathrm{NN}}^n),\quad 0\le n\le K-1.
\end{align*}
\end{theorem}

Following the standard energy-estimate technique for the stabilized BDF1 discretization of the CH equation developed in~\cite{Shen10}, we have the following result.

\begin{theorem}[Unconditional energy stability of GTransNet-BDF1]\label{thm:egy_BDF1}
If the stabilization constant satisfies $\kappa\ge L_f/2$, then the GTransNet-BDF1 scheme~\eqref{GBDF1} is unconditionally energy stable, that is,
\begin{align*}
E(u_{\mathrm{NN}}^{n+1})\le E(u_{\mathrm{NN}}^n),\quad 0\le n\le K-1.
\end{align*}
\end{theorem}

For the GTransNet-BDF2 scheme, the analysis is more involved and yields only conditional energy stability under a time step size restriction, as stated in the following theorem.

\begin{theorem}[Conditional energy stability of GTransNet-BDF2]\label{thm:egy_BDF2}
If the stabilization constant and the time step size satisfy
\begin{align}\label{Dt_constraint}
\kappa\ge L_f,\quad\Delta t\le \frac{8\varepsilon^2}{D L_f^2},
\end{align}
then the GTransNet-BDF2 scheme~\eqref{GBDF2} is energy stable in the sense that
\begin{align}\label{egy_BDF2}
\widetilde E(u_{\mathrm{NN}}^{n+1})\le \widetilde E(u_{\mathrm{NN}}^n),\quad 1\le n\le K-1,
\end{align}
where the modified energy is defined by
$$\widetilde E(u_{\mathrm{NN}}^n):=E(u_{\mathrm{NN}}^n)+\frac{1}{4D\Delta t}\|\delta_tu_{\mathrm{NN}}^{n}\|_{-1}^2+\frac{\kappa+L_f}{2}\|\delta_tu_{\mathrm{NN}}^{n}\|^2.$$
\end{theorem}

\begin{proof}
By the mass conservation of GTransNet-BDF schemes in Theorem~\ref{thm:mass}, $\delta_tu_{\mathrm{NN}}^{n+1}\in L_0^2(\Omega)$, so $(-\Delta)^{-1}\delta_tu_{\mathrm{NN}}^{n+1}$ is well-defined. Taking the $L^2$ inner product of~\eqref{GBDF2a} with $(-\Delta)^{-1}\delta_tu_{\mathrm{NN}}^{n+1}$, applying the divergence theorem together with the boundary conditions, and dividing both sides by $D$, we obtain
\begin{align}\label{egy_bdf2_bdt1}
\frac{1}{2D\Delta t}\bigl(3u_{\mathrm{NN}}^{n+1}-4u_{\mathrm{NN}}^n+u_{\mathrm{NN}}^{n-1},(-\Delta)^{-1}\delta_tu_{\mathrm{NN}}^{n+1}\bigr)=-\bigl(\mu_{\mathrm{NN}}^{n+1},\delta_tu_{\mathrm{NN}}^{n+1}\bigr).
\end{align}
Note that~\eqref{GBDF2b} can be rewritten, using $f_\kappa(u)=f(u)-\kappa u$, as
\begin{align}\label{GBDF2b_rewrite}
\mu_{\mathrm{NN}}^{n+1}=-\varepsilon^2\Delta u_{\mathrm{NN}}^{n+1}+f(2u_{\mathrm{NN}}^n-u_{\mathrm{NN}}^{n-1})+\kappa\delta_{tt}u_{\mathrm{NN}}^{n+1}.
\end{align}
Taking the $L^2$ inner product of~\eqref{GBDF2b_rewrite} with $\delta_tu_{\mathrm{NN}}^{n+1}$ and applying integration by parts, we obtain
\begin{align}\label{egy_bdf2_bdf2}
\bigl(\mu_{\mathrm{NN}}^{n+1},\delta_tu_{\mathrm{NN}}^{n+1}\bigr)=\varepsilon^2\bigl(\nabla u_{\mathrm{NN}}^{n+1},\nabla\delta_tu_{\mathrm{NN}}^{n+1}\bigr)+\bigl(f(2u_{\mathrm{NN}}^n-u_{\mathrm{NN}}^{n-1}),\delta_tu_{\mathrm{NN}}^{n+1}\bigr)+\kappa\bigl(\delta_{tt}u_{\mathrm{NN}}^{n+1},\delta_tu_{\mathrm{NN}}^{n+1}\bigr).
\end{align}
Combining~\eqref{egy_bdf2_bdt1} and~\eqref{egy_bdf2_bdf2} yields
\begin{align}\label{egy_BDF2_master}
\begin{aligned}
&\frac{1}{2D\Delta t}\bigl(3u_{\mathrm{NN}}^{n+1}-4u_{\mathrm{NN}}^n+u_{\mathrm{NN}}^{n-1},(-\Delta)^{-1}\delta_tu_{\mathrm{NN}}^{n+1}\bigr)+\varepsilon^2\bigl(\nabla u_{\mathrm{NN}}^{n+1},\nabla\delta_tu_{\mathrm{NN}}^{n+1}\bigr)\\
&\qquad+\kappa\bigl(\delta_{tt}u_{\mathrm{NN}}^{n+1},\delta_tu_{\mathrm{NN}}^{n+1}\bigr)=-\bigl(f(2u_{\mathrm{NN}}^n-u_{\mathrm{NN}}^{n-1}),\delta_tu_{\mathrm{NN}}^{n+1}\bigr).
\end{aligned}
\end{align}
For the first term in~\eqref{egy_BDF2_master}, by decomposing $3u_{\mathrm{NN}}^{n+1}-4u_{\mathrm{NN}}^n+u_{\mathrm{NN}}^{n-1}=2\delta_tu_{\mathrm{NN}}^{n+1}+\delta_{tt}u_{\mathrm{NN}}^{n+1}$ and applying the identity $a(a-b)=\frac 12[a^2-b^2+(a-b)^2]$ in the $H^{-1}$ inner product, we obtain
\begin{align}\label{LHS_BDF2_first}
\begin{aligned}
&\frac{1}{2D\Delta t}\bigl(3u_{\mathrm{NN}}^{n+1}-4u_{\mathrm{NN}}^n+u_{\mathrm{NN}}^{n-1},(-\Delta)^{-1}\delta_tu_{\mathrm{NN}}^{n+1}\bigr)\\
&\quad=\frac{1}{D\Delta t}\|\delta_tu_{\mathrm{NN}}^{n+1}\|_{-1}^2+\frac{1}{4D\Delta t}\bigl(\|\delta_tu_{\mathrm{NN}}^{n+1}\|_{-1}^2-\|\delta_tu_{\mathrm{NN}}^n\|_{-1}^2+\|\delta_{tt}u_{\mathrm{NN}}^{n+1}\|_{-1}^2\bigr).
\end{aligned}
\end{align}
For the remaining two terms on the left-hand side of~\eqref{egy_BDF2_master}, we again apply the identity $a(a-b)=\frac 12[a^2-b^2+(a-b)^2]$, with $\delta_{tt}u_{\mathrm{NN}}^{n+1}=\delta_tu_{\mathrm{NN}}^{n+1}-\delta_tu_{\mathrm{NN}}^n$, to deduce that
\begin{align}\label{LHS_BDF2_rest}
\begin{aligned}
\varepsilon^2\bigl(\nabla u_{\mathrm{NN}}^{n+1},\nabla\delta_tu_{\mathrm{NN}}^{n+1}\bigr)&=\frac{\varepsilon^2}{2}\bigl(\|\nabla u_{\mathrm{NN}}^{n+1}\|^2-\|\nabla u_{\mathrm{NN}}^n\|^2+\|\nabla\delta_tu_{\mathrm{NN}}^{n+1}\|^2\bigr),\\
\kappa\bigl(\delta_{tt}u_{\mathrm{NN}}^{n+1},\delta_tu_{\mathrm{NN}}^{n+1}\bigr)&=\frac{\kappa}{2}\bigl(\|\delta_tu_{\mathrm{NN}}^{n+1}\|^2-\|\delta_tu_{\mathrm{NN}}^n\|^2+\|\delta_{tt}u_{\mathrm{NN}}^{n+1}\|^2\bigr).
\end{aligned}
\end{align}
For the right-hand side of~\eqref{egy_BDF2_master}, we apply the Taylor expansion of $F$ around $2u_{\mathrm{NN}}^n-u_{\mathrm{NN}}^{n-1}$. Noting that $u_{\mathrm{NN}}^{n+1}-(2u_{\mathrm{NN}}^n-u_{\mathrm{NN}}^{n-1})=\delta_{tt}u_{\mathrm{NN}}^{n+1}$ and $u_{\mathrm{NN}}^n-(2u_{\mathrm{NN}}^n-u_{\mathrm{NN}}^{n-1})=-\delta_tu_{\mathrm{NN}}^n$, the bound~\eqref{f_lipschitz} on $F''=f'$ yields
\begin{align}\label{taylor_F}
\begin{aligned}
F(u_{\mathrm{NN}}^{n+1})&\le F(2u_{\mathrm{NN}}^n-u_{\mathrm{NN}}^{n-1})+f(2u_{\mathrm{NN}}^n-u_{\mathrm{NN}}^{n-1})\,\delta_{tt}u_{\mathrm{NN}}^{n+1}+\frac{L_f}{2}|\delta_{tt}u_{\mathrm{NN}}^{n+1}|^2,\\
F(u_{\mathrm{NN}}^{n})&\ge F(2u_{\mathrm{NN}}^n-u_{\mathrm{NN}}^{n-1})-f(2u_{\mathrm{NN}}^n-u_{\mathrm{NN}}^{n-1})\,\delta_tu_{\mathrm{NN}}^n-\frac{L_f}{2}|\delta_tu_{\mathrm{NN}}^n|^2.
\end{aligned}
\end{align}
Subtracting the second inequality from the first in~\eqref{taylor_F} and using $\delta_{tt}u_{\mathrm{NN}}^{n+1}+\delta_tu_{\mathrm{NN}}^n=\delta_tu_{\mathrm{NN}}^{n+1}$ yields
\begin{align}\label{F_diff_pointwise}
F(u_{\mathrm{NN}}^{n+1})-F(u_{\mathrm{NN}}^n)\le f(2u_{\mathrm{NN}}^n-u_{\mathrm{NN}}^{n-1})\,\delta_tu_{\mathrm{NN}}^{n+1}+\frac{L_f}{2}\bigl(|\delta_{tt}u_{\mathrm{NN}}^{n+1}|^2+|\delta_tu_{\mathrm{NN}}^n|^2\bigr).
\end{align}
Integrating~\eqref{F_diff_pointwise} over $\Omega$ and rearranging gives
\begin{align}\label{nonlinear_est}
-\bigl(f(2u_{\mathrm{NN}}^n-u_{\mathrm{NN}}^{n-1}),\delta_tu_{\mathrm{NN}}^{n+1}\bigr)\le -\bigl(F(u_{\mathrm{NN}}^{n+1})-F(u_{\mathrm{NN}}^n),1\bigr)+\frac{L_f}{2}\bigl(\|\delta_{tt}u_{\mathrm{NN}}^{n+1}\|^2+\|\delta_tu_{\mathrm{NN}}^n\|^2\bigr).
\end{align}
Substituting~\eqref{LHS_BDF2_first}, \eqref{LHS_BDF2_rest}, and~\eqref{nonlinear_est} into~\eqref{egy_BDF2_master} yields
\begin{align}\label{egy_BDF2_combined}
\begin{aligned}
&E(u_{\mathrm{NN}}^{n+1})-E(u_{\mathrm{NN}}^n)+\frac{\varepsilon^2}{2}\|\nabla\delta_tu_{\mathrm{NN}}^{n+1}\|^2+\frac{1}{D\Delta t}\|\delta_tu_{\mathrm{NN}}^{n+1}\|_{-1}^2\\
&\quad+\frac{1}{4D\Delta t}\bigl(\|\delta_tu_{\mathrm{NN}}^{n+1}\|_{-1}^2-\|\delta_tu_{\mathrm{NN}}^n\|_{-1}^2+\|\delta_{tt}u_{\mathrm{NN}}^{n+1}\|_{-1}^2\bigr)\\
&\quad+\frac{\kappa}{2}\bigl(\|\delta_tu_{\mathrm{NN}}^{n+1}\|^2-\|\delta_tu_{\mathrm{NN}}^n\|^2+\|\delta_{tt}u_{\mathrm{NN}}^{n+1}\|^2\bigr)\le \frac{L_f}{2}\bigl(\|\delta_{tt}u_{\mathrm{NN}}^{n+1}\|^2+\|\delta_tu_{\mathrm{NN}}^n\|^2\bigr).
\end{aligned}
\end{align}
Finally, applying the inequality~\eqref{interpolation_young} with $a=1/(D\Delta t)$ and $b=\varepsilon^2/2$ yields, under the time step constraint~\eqref{Dt_constraint},
\begin{align}\label{egy_BDF2_young}
\frac{\varepsilon^2}{2}\|\nabla\delta_tu_{\mathrm{NN}}^{n+1}\|^2+\frac{1}{D\Delta t}\|\delta_tu_{\mathrm{NN}}^{n+1}\|_{-1}^2\ge \varepsilon\sqrt{\frac{2}{D\Delta t}}\,\|\delta_tu_{\mathrm{NN}}^{n+1}\|^2\ge \frac{L_f}{2}\|\delta_tu_{\mathrm{NN}}^{n+1}\|^2.
\end{align}
Combining~\eqref{egy_BDF2_combined} and~\eqref{egy_BDF2_young} with $\kappa\ge L_f$, we arrive at~\eqref{egy_BDF2}.
\end{proof}

\begin{remark}$ $\label{rem:nonlinear_egy}
\begin{itemize}
\item[(i)] The proof of Theorem~\ref{thm:egy_BDF2} partially adapts the energy-estimate framework developed in~\cite{Wang18}, which analyzes a more general family of stabilized second-order semi-implicit BDF schemes for the CH equation. Particularly, in addition to the stabilization term $\kappa(u_{\mathrm{NN}}^{n+1}-2u_{\mathrm{NN}}^{n}+u_{\mathrm{NN}}^{n-1})$, the schemes proposed in~\cite{Wang18} incorporate an additional Laplacian-type stabilization of the form $-\hat{\kappa}\Delta t\Delta(u_{\mathrm{NN}}^{n+1}-u_{\mathrm{NN}}^{n})$ with $\hat{\kappa}>0$ in the chemical potential equation. This extra term contributes to the left-hand side of~\eqref{egy_BDF2_combined} an amount of $\hat{\kappa}\Delta t\|\nabla\delta_tu_{\mathrm{NN}}^{n+1}\|^2$, yielding unconditional energy stability with an appropriate choice of $\hat{\kappa}$.
\item[(ii)] If the second-order extrapolation $f(2u_{\mathrm{NN}}^n-u_{\mathrm{NN}}^{n-1})$ in the GTransNet-BDF2 scheme~\eqref{GBDF2} is replaced by the standard linearization $2f(u_{\mathrm{NN}}^n)-f(u_{\mathrm{NN}}^{n-1})$, an analogous energy stability result holds only under the more restrictive time step constraint $\Delta t\le 8\varepsilon^2/(9DL_f^2)$ (cf.~\cite{Shen10,Wang18}). Moreover, this constraint cannot be relaxed by increasing the stabilization constant $\kappa$. The form $f(2u_{\mathrm{NN}}^n-u_{\mathrm{NN}}^{n-1})$ adopted here produces the term $\|\delta_{tt}u_{\mathrm{NN}}^{n+1}\|^2$ in the nonlinear estimate~\eqref{nonlinear_est}, which can be absorbed by the corresponding stabilization term in~\eqref{egy_BDF2_combined} provided that $\kappa\ge L_f$, thereby relaxing the time step constraint by a factor of $9$.
\end{itemize}
\end{remark}

\section{Discrete mass-conserving projection}\label{sec:projection}

While Theorem~\ref{thm:mass} guarantees mass conservation in the space-continuous sense, the fully discrete mass computed from $\bm c_{\mathrm{unc}}^{n+1}=\bm R^{-1}\bm Q^T\bm b^{n+1}$ (cf.~\eqref{eq:normal}) may not be conserved exactly, since the over-determined least-squares system~\eqref{LSQ} generally has nonzero residuals at the collocation points. In this section, we describe a post-processing projection that enforces exact discrete mass conservation with minimal perturbation of the least-squares approximation, and establish the key properties of this projection that justify its use in long-time simulations.

For all numerical computations and for the projection introduced below, we denote by $\mathcal M^n$ an approximation of $\mathcal M(u_{\mathrm{NN}}^n)$, i.e.,
\begin{align*}
\mathcal M(u_{\mathrm{NN}}^n)&\approx\mathcal M^n:=\sum_{q}\omega_q\,u_{\mathrm{NN}}^n(\bm x_q),
\end{align*}
where $\{\bm x_q,\omega_q\}$ denotes a set of quadrature points and weights; in the numerical experiments, we use the tensor-product two-point Gauss-Legendre quadrature rule. Substituting the basis expansion $u_{\mathrm{NN}}^n(\bm x)=\sum_{j=1}^{N_L}\alpha_j^n\phi_j(\bm x)$ into the definition of $\mathcal M^n$ gives the following identity
\begin{center}
$\mathcal M^n=\bm g^T\bm\alpha^n$,\; where $\bm g=(g_1,\ldots,g_{N_L})^T\in\mathbb R^{N_L}$ and $g_j=\displaystyle\sum_q\omega_q\phi_j(\bm x_q),\;1\le j\le N_L$.
\end{center}
With $\hat{\bm g}:=(\bm g^T,\bm 0^T)^T\in\mathbb R^{2N_L}$, the discrete mass-conservation requirement $\mathcal M^{n+1}=\mathcal M^0$ becomes the single linear constraint
\begin{align}\label{cond:constraint}
\hat{\bm g}^T\bm c^{n+1}=\mathcal M^0,\;\text{ where }\bm c^{n+1}=((\bm\alpha^{n+1})^T,(\bm\beta^{n+1})^T)^T.
\end{align}

\subsection{Constrained least-squares formulation}

The unconstrained solution $\bm c_{\mathrm{unc}}^{n+1}$ approximates the discretized PDE in the least-squares sense and satisfies~\eqref{cond:constraint} up to a small residual at the level of the least-squares error. We therefore correct $\bm c_{\mathrm{unc}}^{n+1}$ by the smallest perturbation, in the weighted norm induced by $\bm R^T\bm R$, that exactly restores~\eqref{cond:constraint}:
\begin{align}\label{mass_minimization}
\bm c^{n+1}=\arg\min_{\bm c\in\mathbb R^{2N_L}}\|\bm c-\bm c_{\mathrm{unc}}^{n+1}\|_{\bm R^T\bm R}^2\quad\text{subject to}\quad \hat{\bm g}^T\bm c=\mathcal M^0,
\end{align}
where $\|\bm v\|_{\bm R^T\bm R}^2=\bm v^T\bm R^T\bm R\bm v=\|\bm R\bm v\|_2^2$ for all $\bm v\in\mathbb R^{2N_L}$. The choice of the $\bm R^T\bm R$-norm is natural for least-squares problems: $\|\bm R(\bm c-\bm c_{\mathrm{unc}}^{n+1})\|_2$ measures exactly the increase in the least-squares residual norm caused by the correction, so among all feasible coefficient vectors the minimizer $\bm c^{n+1}$ is the one that least perturbs the discretized-PDE fit. The associated Lagrangian
\begin{align*}
\mathcal L(\bm c,\lambda)=\frac{1}{2}\|\bm c-\bm c_{\mathrm{unc}}^{n+1}\|_{\bm R^T\bm R}^2+\lambda\bigl(\hat{\bm g}^T\bm c-\mathcal M^0\bigr)
\end{align*}
yields the Karush-Kuhn-Tucker (KKT) system $\bm R^T\bm R(\bm c-\bm c_{\mathrm{unc}}^{n+1})+\lambda\hat{\bm g}=\bm 0$ together with the mass constraint $\hat{\bm g}^T\bm c=\mathcal M^0$, from which the closed form
\begin{align}\label{mass_proj}
\bm c^{n+1}=\bm c_{\mathrm{unc}}^{n+1}-\frac{\hat{\bm g}^T\bm c_{\mathrm{unc}}^{n+1}-\mathcal M^0}{\hat{\bm g}^T\bm z}\,\bm z,\qquad \bm z:=(\bm R^T\bm R)^{-1}\hat{\bm g},
\end{align}
follows. In~\eqref{mass_proj}, the vector $\bm z\in\mathbb R^{2N_L}$ is computed once for all time steps by two triangular solves (one with $\bm R^T$ and one with $\bm R$); and the projection then requires only one inner product and a scalar correction along $\bm z$, adding $\mathcal O(N_L)$ operations per time step. Since this cost is dominated by the matrix-vector product $\bm Q^T\bm b^{n+1}$ and back substitution when computing $\bm c_{\mathrm{unc}}^{n+1}$, the projection adds negligible computational cost.

\subsection{Properties of the projection}\label{subsec:proj_props}

Let
$
H_{\mathcal M^0}:=\bigl\{\bm c\in\mathbb R^{2N_L}:\hat{\bm g}^T\bm c=\mathcal M^0\bigr\}
$
denote the affine hyperplane of coefficient vectors that satisfy the discrete mass constraint~\eqref{cond:constraint}, and define the projection operator $\mathcal P:\mathbb R^{2N_L}\to H_{\mathcal M^0}$ by
\begin{align}\label{proj_op}
\mathcal P(\bm c):=\bm c-\frac{\hat{\bm g}^T\bm c-\mathcal M^0}{\hat{\bm g}^T\bm z}\,\bm z,\quad \forall\,\bm c\in \mathbb R^{2N_L},
\end{align}
so that $\bm c^{n+1}=\mathcal P(\bm c_{\mathrm{unc}}^{n+1})$ (cf.~\eqref{mass_proj}). Note that $\hat{\bm g}^T\bm z=\hat{\bm g}^T(\bm R^T\bm R)^{-1}\hat{\bm g}>0$ since $\bm R^T\bm R$ is symmetric positive definite and $\hat{\bm g}\ne\bm 0$, so $\mathcal P$ is well defined. The map $\mathcal P$ possesses the following properties:

\paragraph{(P1) Exact discrete mass conservation} A direct computation gives
\begin{align*}
\hat{\bm g}^T\mathcal P(\bm c)=\hat{\bm g}^T\bm c-(\hat{\bm g}^T\bm c-\mathcal M^0)=\mathcal M^0,\quad\forall\,\bm c\in\mathbb R^{2N_L}.
\end{align*}
Therefore, the projected coefficient vector $\bm c^{n+1}$ satisfies mass conservation $\mathcal M^{n+1}=\hat{\bm g}^T\bm c^{n+1}=\hat{\bm g}^T\mathcal P(\bm c_{\mathrm{unc}}^{n+1})=\mathcal M^0$ at every time step.

\paragraph{(P2) Idempotency} If $\bm c\in H_{\mathcal M^0}$, then $\hat{\bm g}^T\bm c=\mathcal M^0$, so $\mathcal P(\bm c)=\bm c$ due to~\eqref{proj_op}. Since $\mathcal P(\bm c)\in H_{\mathcal M^0}$ for all $\bm c\in \mathbb R^{2N_L}$ by (P1), we have $\mathcal P\circ\mathcal P=\mathcal P$.

\paragraph{(P3) Non-expansiveness in the $\bm R^T\bm R$-norm} For all $\bm c,\bm c'\in\mathbb R^{2N_L}$,
\begin{align}\label{P_nonexp}
\|\mathcal P(\bm c)-\mathcal P(\bm c')\|_{\bm R^T\bm R}\le \|\bm c-\bm c'\|_{\bm R^T\bm R}.
\end{align}
Indeed, $\mathcal P(\bm c)-\mathcal P(\bm c')$ is the $\bm R^T\bm R$-orthogonal projection of $\bm c-\bm c'$ onto the linear subspace $\hat{\bm g}^\perp=\{\bm v\in\mathbb R^{2N_L}:\hat{\bm g}^T\bm v=0\}$, and the standard Pythagorean identity for orthogonal projections gives
\begin{align*}
\|\bm c-\bm c'\|_{\bm R^T\bm R}^2=\|\mathcal P(\bm c)-\mathcal P(\bm c')\|_{\bm R^T\bm R}^2+\frac{\bigl(\hat{\bm g}^T(\bm c-\bm c')\bigr)^2}{\hat{\bm g}^T\bm z}.
\end{align*}
In particular, taking $\bm c=\bm c_{\mathrm{unc}}^{n+1}$ and $\bm c'=\bm c^n$ in~\eqref{P_nonexp}, noting that $\bm c^n=\mathcal P(\bm c_{\mathrm{unc}}^{n})\in H_{\mathcal M^0}$ by (P1) and $\mathcal P(\bm c^n)=\bm c^n$ by (P2), we obtain
$$\|\bm c^{n+1}-\bm c^n\|_{\bm R^T\bm R}\le \|\bm c_{\mathrm{unc}}^{n+1}-\bm c^n\|_{\bm R^T\bm R}.$$

\paragraph{(P4) Bounded perturbation} Substituting~\eqref{mass_proj} into the objective function of~\eqref{mass_minimization} and using the identity $\bm z^T\bm R^T\bm R\bm z=\hat{\bm g}^T\bm z$ (which follows from $\bm z=(\bm R^T\bm R)^{-1}\hat{\bm g}$), the optimal value is
\begin{align*}
\|\bm c^{n+1}-\bm c_{\mathrm{unc}}^{n+1}\|_{\bm R^T\bm R}^2=\frac{\bigl(\hat{\bm g}^T\bm c_{\mathrm{unc}}^{n+1}-\mathcal M^0\bigr)^2}{\hat{\bm g}^T\bm z},
\end{align*}
i.e., the perturbation $\|\bm c^{n+1}-\bm c_{\mathrm{unc}}^{n+1}\|_{\bm R^T\bm R}$ is proportional to the absolute mass error $|\hat{\bm g}^T\bm c_{\mathrm{unc}}^{n+1}-\mathcal M^0|$ of the unconstrained solution, and vanishes when $\bm c_{\mathrm{unc}}^{n+1}$ already satisfies the mass constraint.

\medskip
Together, properties (P1)--(P4) show that the projection~\eqref{mass_proj} enforces exact discrete mass conservation while introducing the smallest possible correction in the $\bm R^T\bm R$-norm. In contrast to a global a posteriori approach that uniformly shifts the solution to match the target mass and ignores the PDE residual, the proposed projection~\eqref{mass_proj} corrects $\bm c_{\mathrm{unc}}^{n+1}$ along the optimal direction $\bm z$ that minimizes the disturbance to this residual. The projection is applied at every time step in the numerical experiments of Section~\ref{sec:num}. We refer to~\cite{Li24AC} for a maximum-bound-principle-preserving and mass-conservative projection method for the conservative Allen–Cahn equation and to~\cite{Qiao25CHsurf} for a mass-projection strategy for the CH equation on surfaces.

\section{Numerical experiments}\label{sec:num}
In this section, we demonstrate the performance of the proposed GTransNet-BDF2 scheme~\eqref{GBDF2} coupled with the mass-conserving projection (cf. Section~\ref{sec:projection}) on a range of test cases. First, for the two-dimensional CH equation, we verify the temporal convergence of the scheme via a manufactured solution. Through the shape relaxation and coarsening dynamics tests, we then illustrate that the scheme reproduces the expected evolution while preserving mass conservation and energy dissipation. The method is further applied to solve the CH equation on irregular domains, including circular and amoeba-shaped ones. Second, for the three-dimensional CH equation, we confirm the temporal convergence of the scheme and perform a coarsening dynamics test, which demonstrates its robustness in preserving the two intrinsic properties over long-time simulations. Finally, we extend the GTransNet-BDF framework to the CH equation with variable mobility (see~\ref{sec:appendix_varmob} for the detailed formulation and implementation) and show its numerical performance for the degenerate mobility $M(u)=|1-u^2|$.

Unless otherwise stated, we consider the standard double-well potential $F(u)=\tfrac{1}{4}(u^2-1)^2$, with $f(u)=F'(u)=u^3-u$ and wells at $u=\pm 1$. The interior collocation points $\{\bm x_i\}_{i=1}^{K_{\mathrm{int}}}$ are chosen according to the boundary condition: a uniform vertex grid on $\overline{\Omega}$ in the periodic case, and the cell centers of a uniform partition of $\Omega$ in the homogeneous Neumann counterpart; for irregular domains, we apply the same uniform partition to a bounding box of $\Omega$ and retain only the cell centers inside $\Omega$. The boundary collocation points $\{\bm x_i^{\mathrm{bd}}\}_{i=1}^{K_{\mathrm{bd}}}$ are equally spaced for rectangular domains, and are placed at equally spaced parameter values along the parametrized boundaries of irregular domains. In addition, the GTransNet hidden-layer neurons are generated on a covering ball $B_R(\bm x_c)\supset\Omega$ (cf.~Remark~\ref{rem:affine_shift}), where $\bm x_c$ is the center of $\Omega$ (or of its bounding box), and the radius $R$, specified for each example, is chosen so that $B_R(\bm x_c)$ is slightly larger than $\Omega$.

We employ the GTransNet basis with $L=2$ hidden layers in most test cases. As discussed in~\cite{Cheng26}, this choice achieves a favorable balance between accuracy and computational efficiency. Adding more hidden layers (e.g., $L=3$) may offer some accuracy improvements, but at the expense of increased computational cost in assembling the least-squares system due to the evaluation of higher-order derivatives via the chain rule. The case $L=3$ is included only in the selected convergence tests for comparison.

\subsection{Two-dimensional Cahn-Hilliard equation}\label{subsec:num_2D}

\subsubsection{Convergence test}
We set the domain $\Omega=(-0.5,0.5)^2$ and the terminal time $T=1$. By adding an external forcing term to the right-hand side of the first equation in~\eqref{eq:CHmix_kappa}, we take the exact solution to be
\begin{align*}
    u(x,y,t)=\sin(2\pi x)\cos(2\pi y)\cos(t),
\end{align*}
subject to the periodic boundary conditions. We fix $D=0.5$, $\kappa=0$, $\delta=0.5$, $R=0.75$, and vary the interfacial width $\varepsilon\in\{0.2,0.1\}$. The collocation points are sampled from a uniform $n_x\times n_y$ mesh, and the test data consist of $4n_xn_y$ uniformly distributed random points in $\Omega$. The parameters used for each value of $\varepsilon$ are listed in Table~\ref{table:params}.

\begin{table}[!ht]
    \centering
    \small
    \setlength{\extrarowheight}{3pt}
\subcaptionbox{$L=2$ hidden layers}[.49\linewidth]{
    \begin{tabular}{|c|c|c|c|c|c|}
    \hline
    $\varepsilon$ & $N_1$ & $N_2$ & $\gamma$ & $n_x=n_y$ \\
    \hline
    $0.2$ & $600$ & $500$ & $3$ & $50$ \\
    $0.1$ & $1500$&$1000$&  $4$   &  $80$\\
    \hline
    \end{tabular}
}\hspace{.05cm}
\subcaptionbox{$L=3$ hidden layers}[.49\linewidth]{
    \begin{tabular}{|c|c|c|c|c|c|c|}
    \hline
    $\varepsilon$  & $N_1$ & $N_2$ & $N_3$ & $\gamma$ & $n_x=n_y$ \\
    \hline
    $0.2$ &600  &500  &500  &2.8  &50  \\
    $0.1$ &1500  & 1000 & 1000 & 4 & 80 \\
    \hline
    \end{tabular}
}
    \caption{[2D convergence test] Parameter settings for $\varepsilon\in\{0.2,0.1\}$.}
    \label{table:params}
\end{table}

Table~\ref{table:convergence} presents $L^{\infty}$ errors of the phase variable at the final time by the GTransNet-BDF2 scheme with $L=2$ and $L=3$ hidden layers for $\varepsilon\in\{0.2,0.1\}$. In all cases, the scheme achieves the expected second-order convergence in time. Moreover, the errors for $L=2$ and $L=3$ are nearly identical, indicating that the spatial approximation error with two hidden layers is already negligible relative to the temporal error. We note that similar convergence results hold for the homogeneous Neumann boundary conditions; a corresponding three-dimensional convergence test under the same boundary conditions is presented in Section~\ref{subsec:num_3D}.

\begin{table}[!ht]
    \centering
    \small
    \setlength{\extrarowheight}{3pt}
\subcaptionbox{$\varepsilon=0.2$}[.49\linewidth]{
\resizebox{.49\textwidth}{!}{
\begin{tabular}{|c|c|c|}
\hline
\multirow{2}{*}{$\Delta t$} & GTransNet-BDF2 & GTransNet-BDF2 \\
 & $L=2$ hidden layers & $L=3$ hidden layers\\ \hline
$1/10$  &1.01e-03 	   &1.01e-03 \\
$1/20$  &2.40e-04 [2.07]&2.40e-04 [2.07]\\
$1/40$  &5.82e-05 [2.04]&5.83e-05 [2.04]\\
$1/80$  &1.44e-05 [2.02]&1.44e-05 [2.02]\\
$1/160$ &3.60e-06 [2.00]&3.64e-06 [1.98]\\
$1/320$ &9.09e-07 [1.99]&9.32e-07 [1.97]\\
\hline
\end{tabular}
}
}\hspace{.05cm}
\subcaptionbox{$\varepsilon=0.1$}[.49\linewidth]{
\resizebox{.49\textwidth}{!}{
\begin{tabular}{|c|c|c|}
\hline
\multirow{2}{*}{$\Delta t$} & GTransNet-BDF2 & GTransNet-BDF2 \\
 & $L=2$ hidden layers & $L=3$ hidden layers \\ \hline
$1/10$  &5.87e-03&5.87e-03\\
$1/20$  &1.54e-03 [1.93]&1.54e-03 [1.96]\\
$1/40$  &3.84e-04 [2.00]&3.84e-04 [2.00]\\
$1/80$  &9.54e-05 [2.01]&9.54e-05 [2.01]\\
$1/160$ &2.37e-05 [2.01]&2.38e-05 [2.00]\\
$1/320$ &6.08e-06 [1.96]&6.18e-06 [1.95]\\
\hline
\end{tabular}
}
}
    \caption{[2D convergence test] $L^{\infty}$ errors of the phase variable at the final time for $\varepsilon\in\{0.2,0.1\}$.}
    \label{table:convergence}
\end{table}

\subsubsection{Shape relaxation}
We next apply the GTransNet-BDF2 scheme with 2 hidden layers to a shape relaxation test case~\cite{Ham22,Doan25b} under homogeneous Neumann boundary conditions. We set $\Omega=(0,1)^2$, $D=0.05$, $\kappa=2$, $\delta=0.5$, $R=0.75$, and $\varepsilon\in\{0.02,0.01\}$. The initial condition has a value of 1 inside the square $[0.25,0.75]^2$ and $-1$ outside.

For $\varepsilon=0.02$, we run the simulation until $T=1.5$ using $N_1=2000$, $N_2=1000$, $\gamma=12$, $\Delta t=\num{1e-2}$, $K_{\mathrm{int}}=150^2$ interior and $K_{\mathrm{bd}}=600$ boundary collocation points. For $\varepsilon=0.01$, we run until $T=3$ with $N_1=3000$, $N_2=1500$, $\gamma=16$, $\Delta t=\num{5e-3}$, $K_{\mathrm{int}}=200^2$, and $K_{\mathrm{bd}}=800$.

\begin{figure}[!ht]
    \centering
\includegraphics[width=0.3\linewidth]{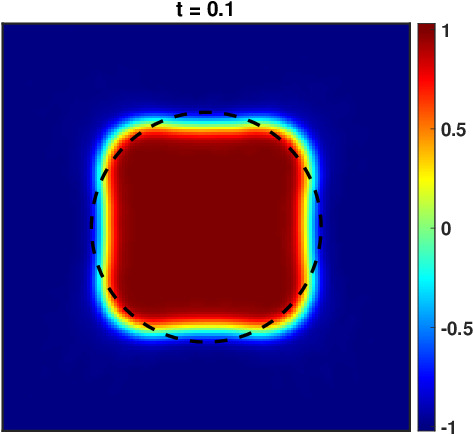}\qquad
\includegraphics[width=0.3\linewidth]{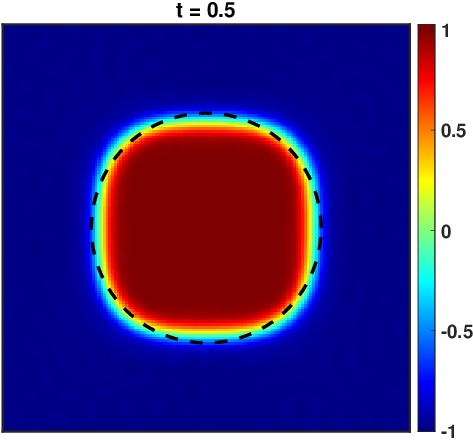}\qquad
\includegraphics[width=0.3\linewidth]{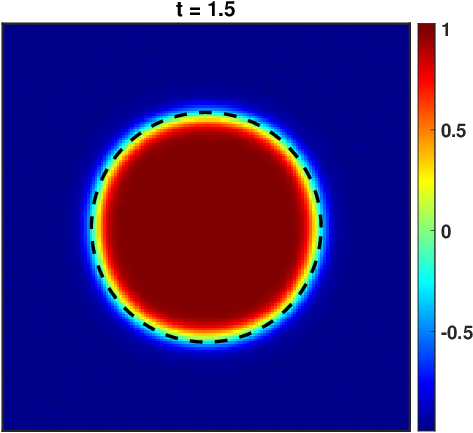}

\vspace{.4cm}

\includegraphics[width=0.3\linewidth]{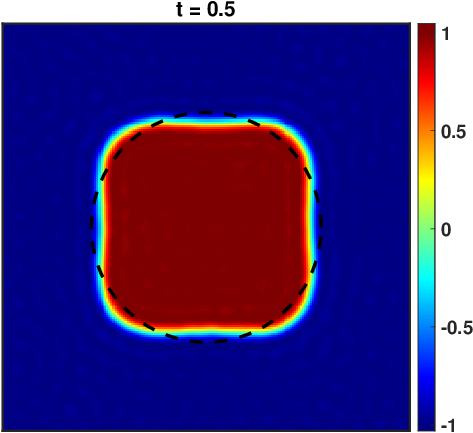}\qquad
\includegraphics[width=0.3\linewidth]{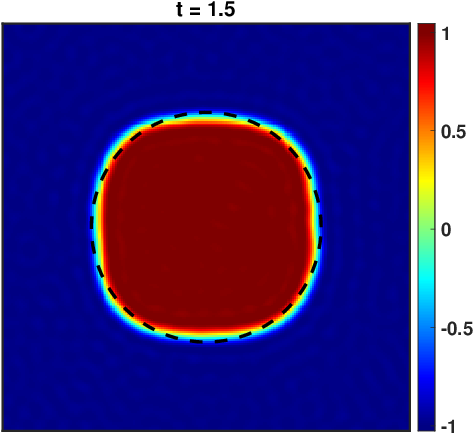}\qquad
\includegraphics[width=0.3\linewidth]{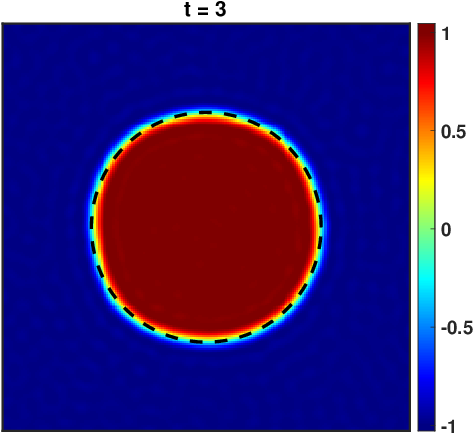}
    \caption{[Shape relaxation] Snapshots of the phase variable by the GTransNet-BDF2 scheme. Top: $\varepsilon=0.02$ at $t=0.1, 0.5$, and $1.5$. Bottom: $\varepsilon=0.01$ at $t=0.5, 1.5$, and $3$.}
    \label{fig:relax_snapshots}
\end{figure}

\begin{figure}[!ht]
    \centering
\includegraphics[width=0.35\linewidth]{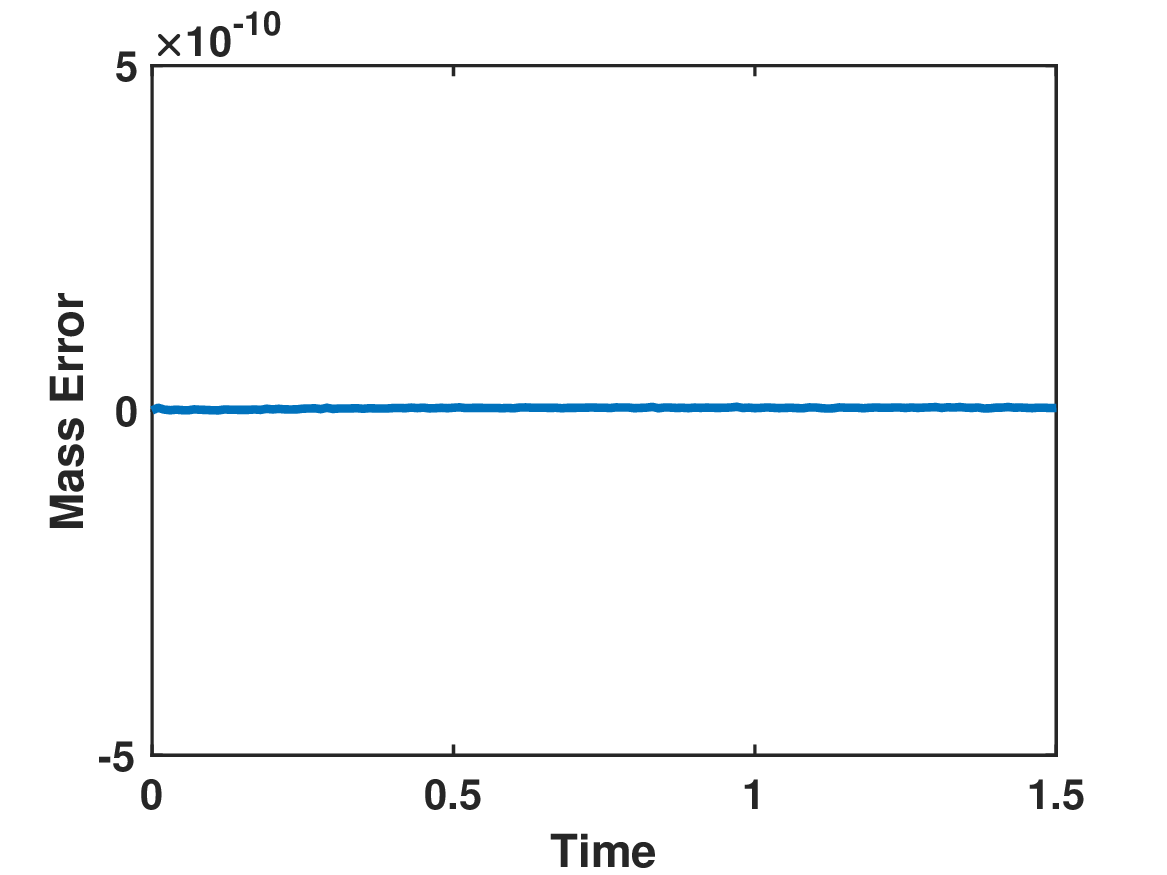}
\includegraphics[width=0.35\linewidth]{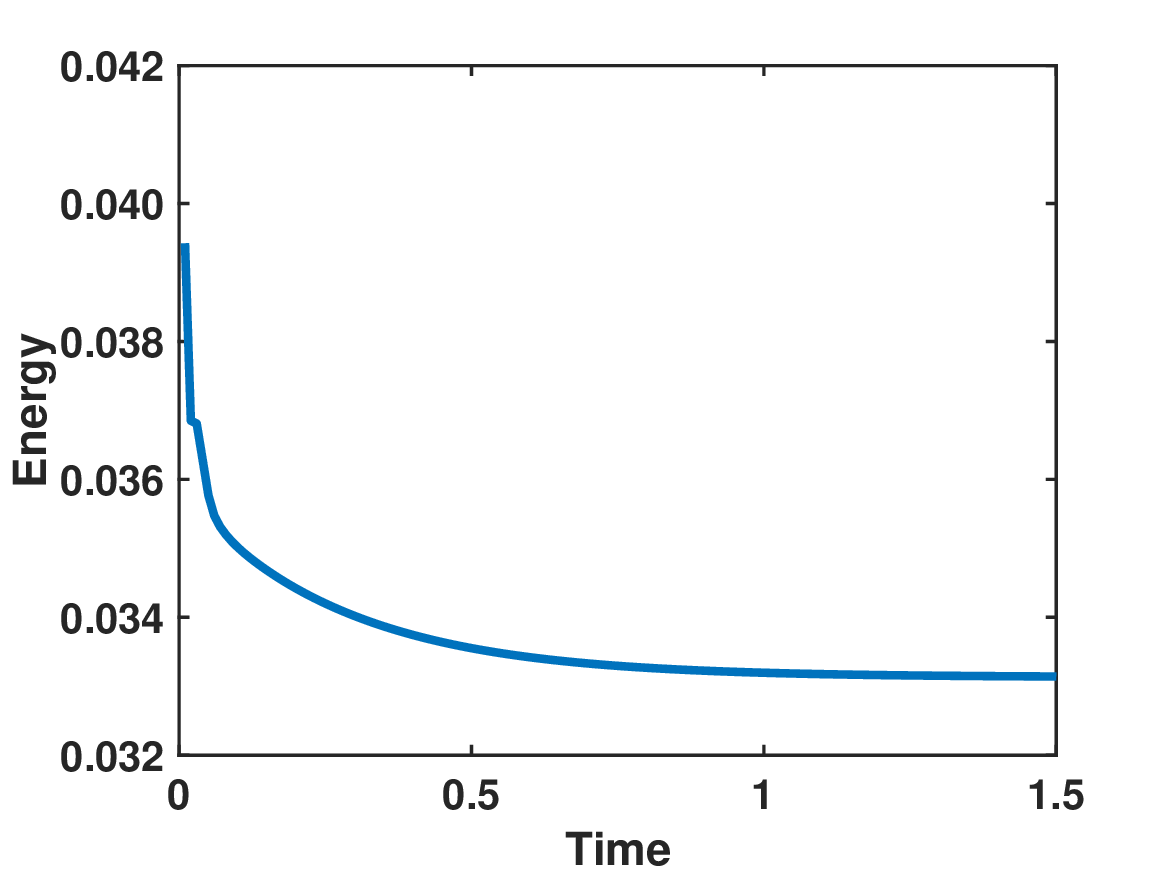}

\includegraphics[width=0.35\linewidth]{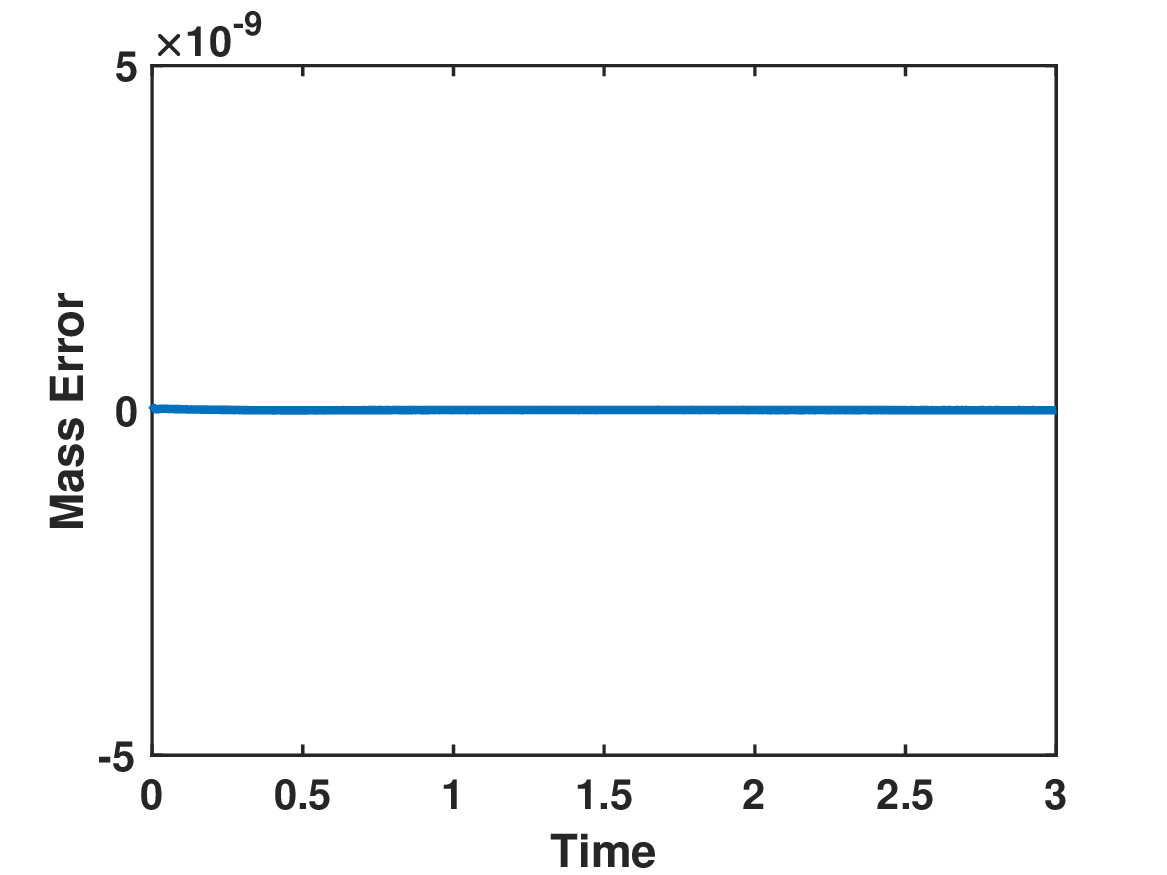}
\includegraphics[width=0.35\linewidth]{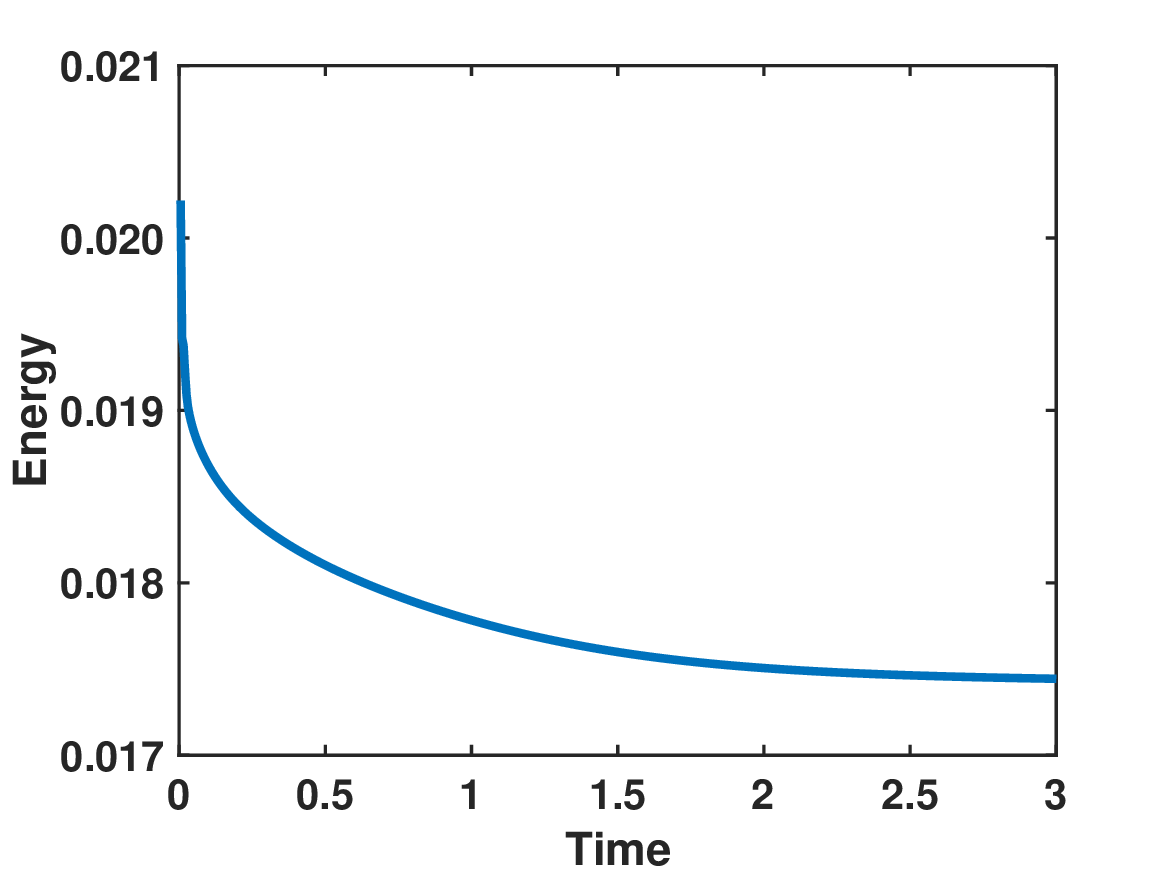}
    \caption{[Shape relaxation] Evolution of the absolute mass error and free energy by the GTransNet-BDF2 scheme with $\varepsilon=0.02$ (top) and $\varepsilon=0.01$ (bottom).}
    \label{fig:relax_mass_energy}
\end{figure}

Snapshots of the phase variable at different times for $\varepsilon\in\{0.02,0.01\}$ are shown in Figure~\ref{fig:relax_snapshots}, where the dashed circles, centered at $(0.5,0.5)$ with radius $\frac{1}{2\sqrt{\pi}}$, indicate the interface at steady state. We observe that the initial square gradually transitions into a circle that closely matches the dashed reference, with the circular shape forming at an earlier stage for $\varepsilon = 0.02$ than for $\varepsilon = 0.01$. Figure~\ref{fig:relax_mass_energy} confirms mass conservation and energy dissipation of the numerical solutions, with the former ensured by the projection introduced in Section~\ref{sec:projection}.

\subsubsection{Coarsening dynamics}\label{subsec:coarsening}
We simulate the long-time coarsening dynamics of the phase separation process with different volume fractions in $\Omega=(0,1)^2$ under homogeneous Neumann boundary conditions. The initial condition is a small random perturbation of the uniform state
$$u_0(x,y) = \overline{u}+0.05\,\mathrm{rand}(x,y),$$
where $\mathrm{rand}(x,y)$ is uniformly distributed in $[-1,1]$. We fix $\varepsilon=0.02$, $D=0.05$, $\kappa=2$, $\delta=0.5$, $R=0.75$, $K_{\mathrm{int}}=250^2$, $K_{\mathrm{bd}}=1000$, and consider $\overline{u}\in\{0,0.5\}$. The simulation is run until $T=40$ using the GTransNet-BDF2 scheme with $N_1=2000$, $N_2=1000$, $\gamma=12$, $\Delta t=\num{1e-3}$ for $\overline{u}=0$, and $N_1=3000$, $N_2=1500$, $\gamma=14$, $\Delta t=\num{5e-3}$ for $\overline{u}=0.5$.

\begin{figure}[!ht]
    \centering
\includegraphics[width=0.3\linewidth]{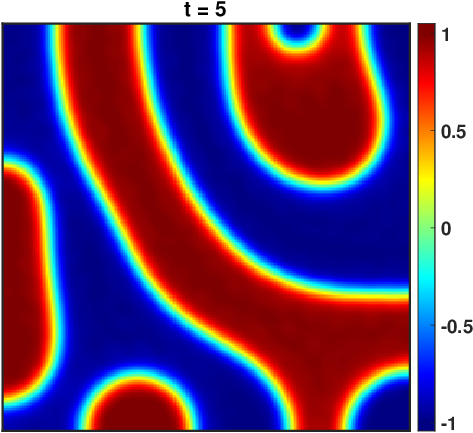}\qquad
\includegraphics[width=0.3\linewidth]{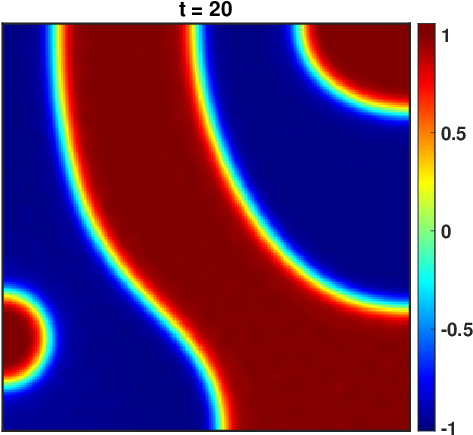}\qquad
\includegraphics[width=0.3\linewidth]{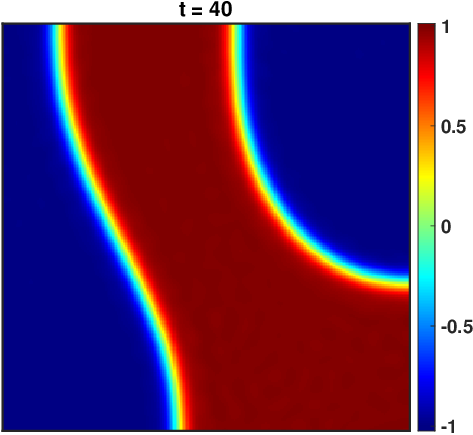}

\vspace{.4cm}

\includegraphics[width=0.3\linewidth]{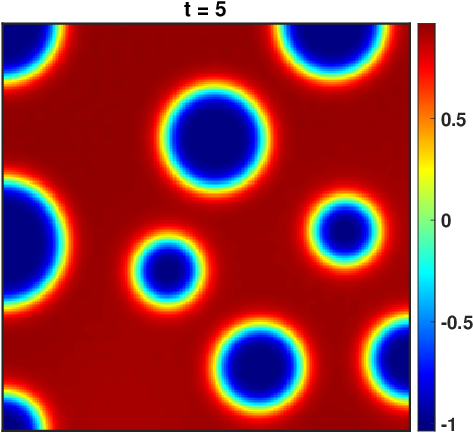}\qquad
\includegraphics[width=0.3\linewidth]{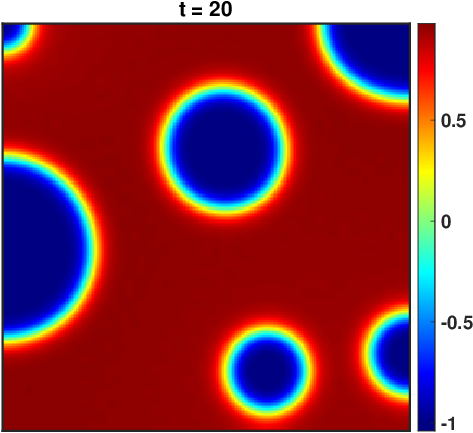}\qquad
\includegraphics[width=0.3\linewidth]{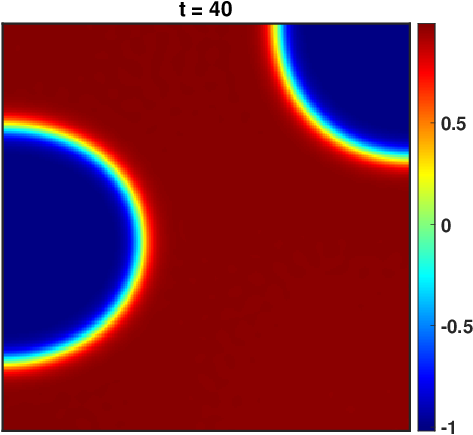}
    \caption{[2D coarsening dynamics] Snapshots of the phase variable by the GTransNet-BDF2 scheme with $\overline{u}=0$ (top) and $\overline{u}=0.5$ (bottom) at $t=5,20$, and $40$.}
    \label{fig:coarsening_eps002_snapshots}
\end{figure}

\begin{figure}[!ht]
    \centering
\includegraphics[width=0.35\linewidth]{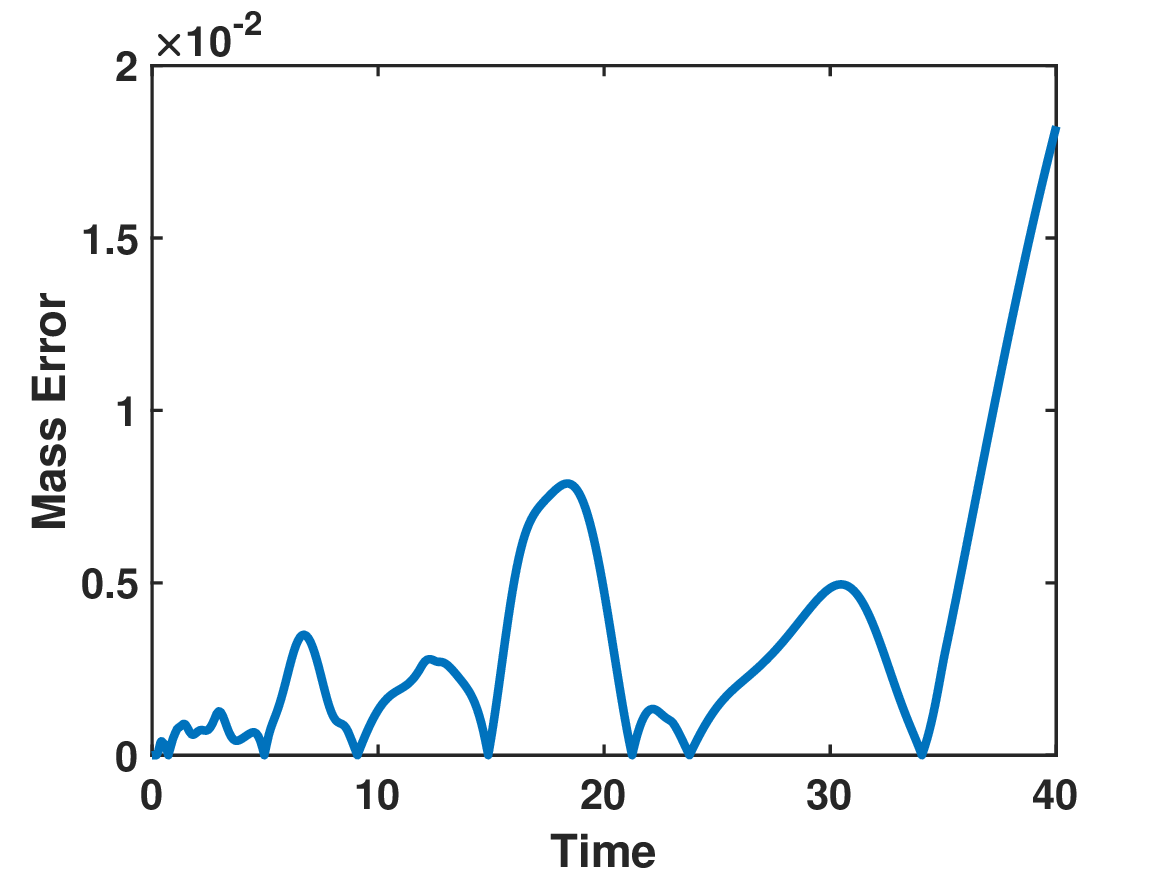}
\includegraphics[width=0.35\linewidth]{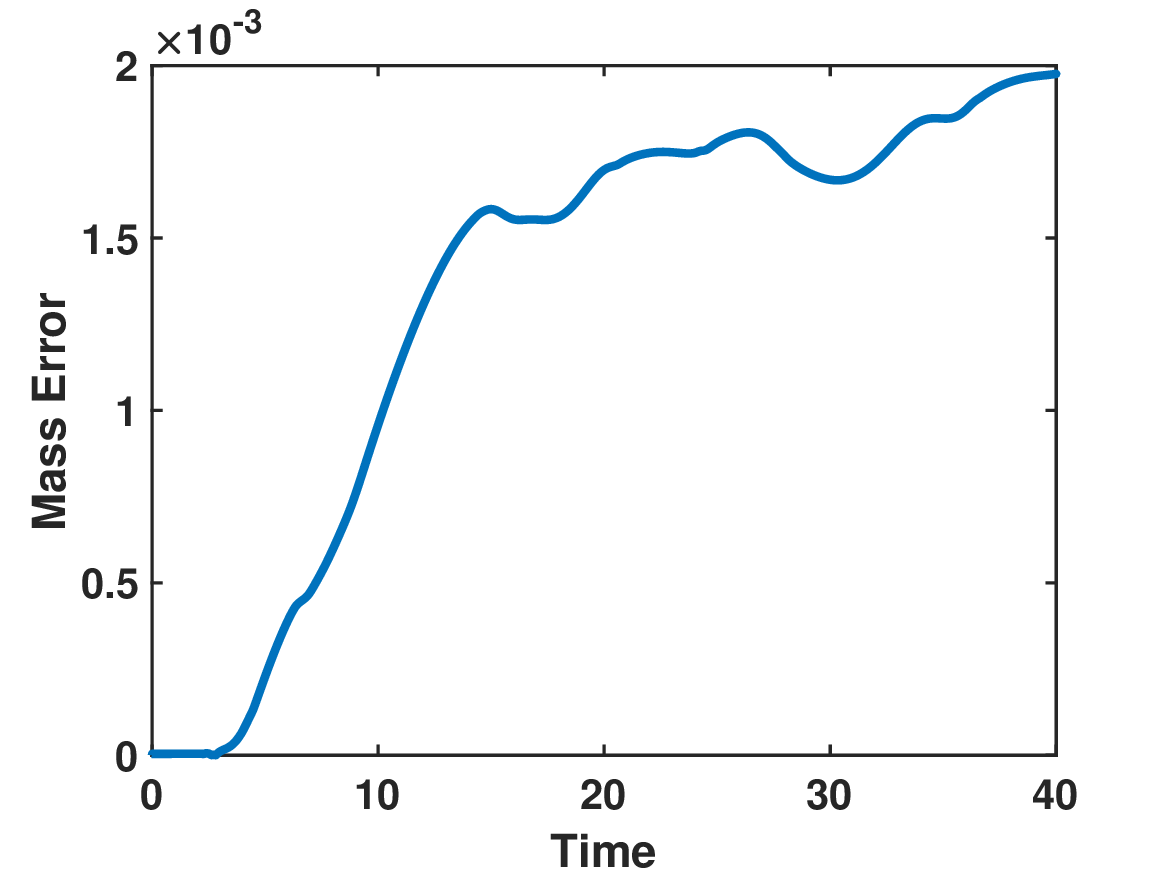}

\includegraphics[width=0.35\linewidth]{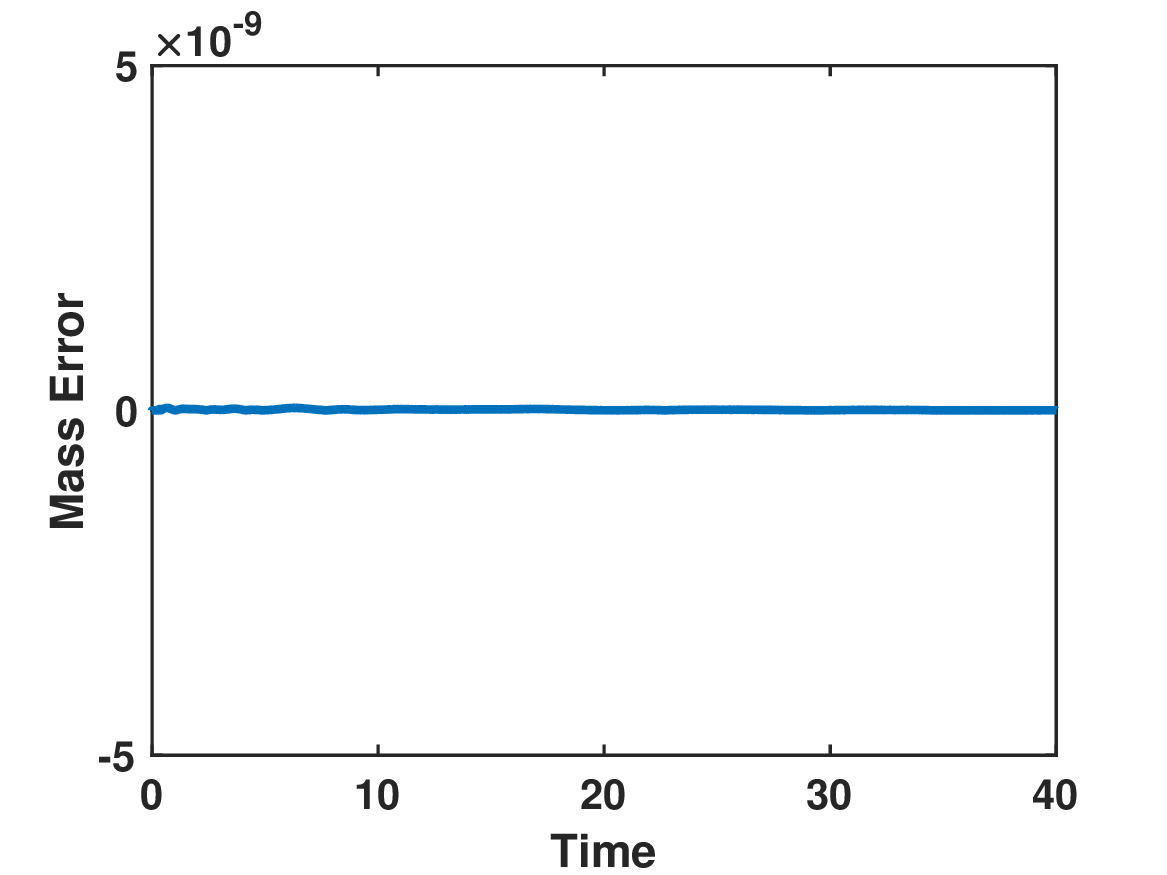}
\includegraphics[width=0.35\linewidth]{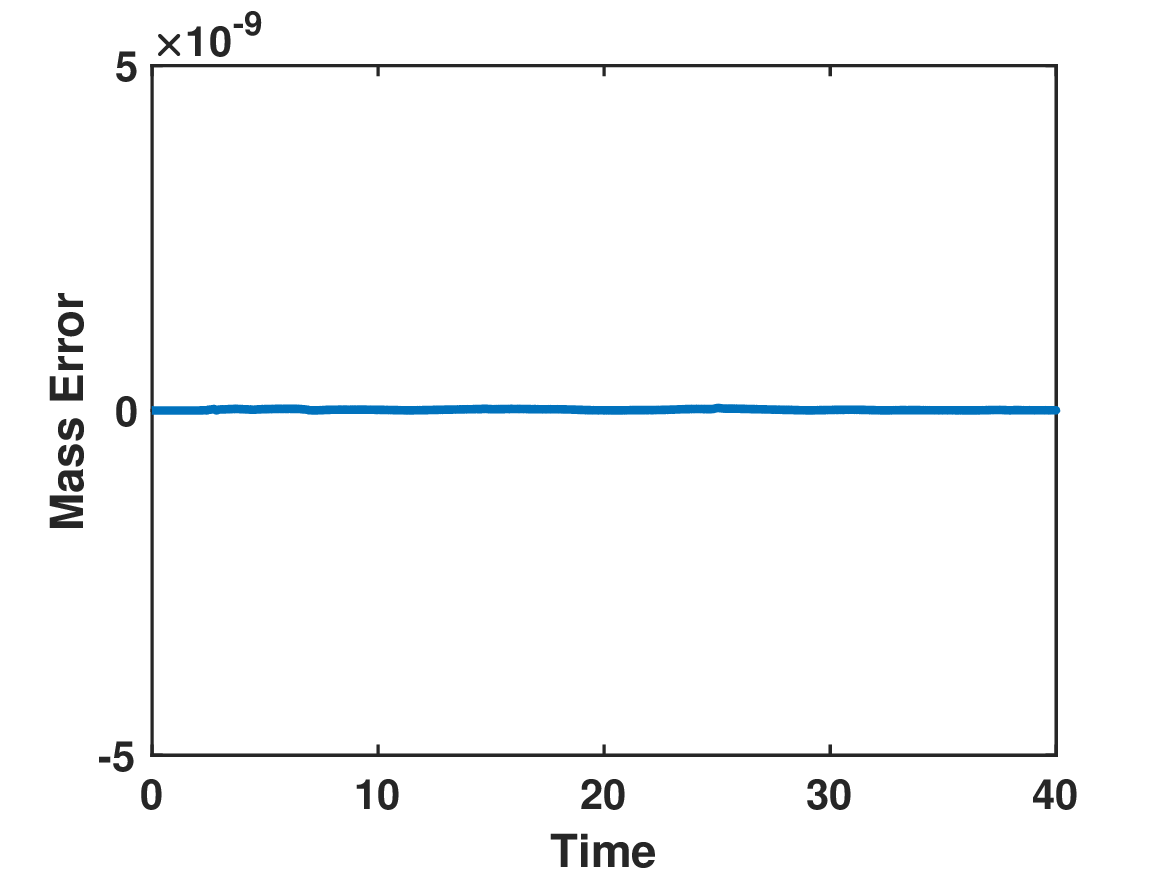}
    \caption{[2D coarsening dynamics] Evolution of the absolute mass error by the GTransNet-BDF2 scheme without the mass-conserving projection (top) and with the projection (bottom), for $\overline{u}=0$ (left) and $\overline{u}=0.5$ (right).}
    \label{fig:coarsening_eps002_mass}
\end{figure}

\begin{figure}[!ht]
    \centering
\includegraphics[width=0.35\linewidth]{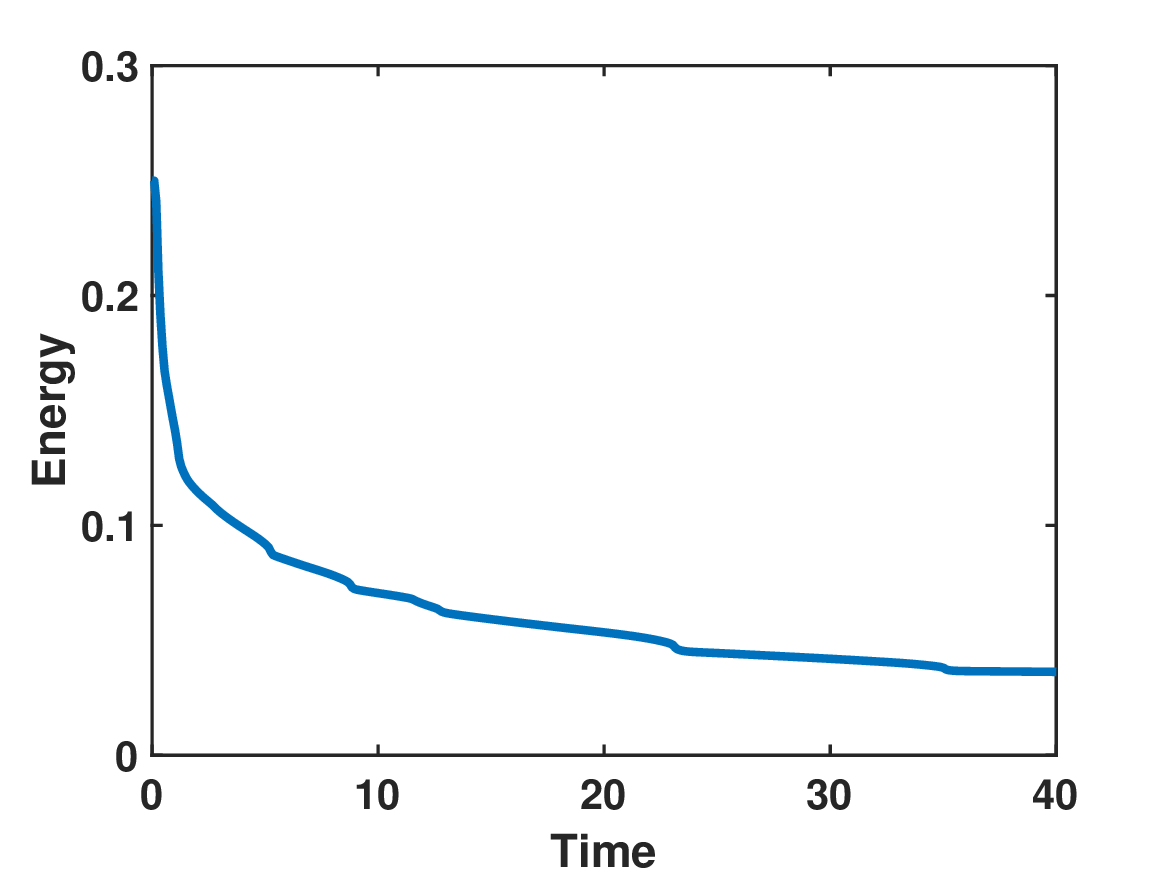}
\includegraphics[width=0.35\linewidth]{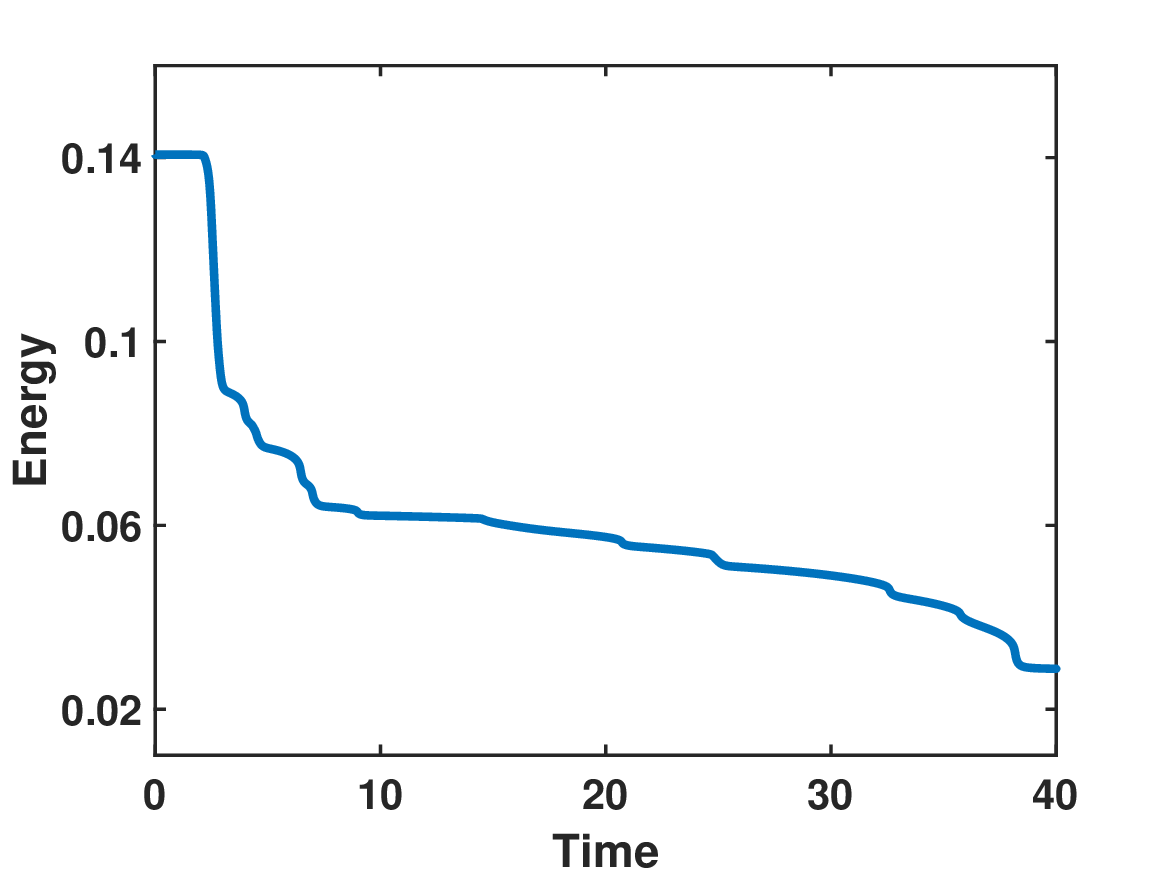}
    \caption{[2D coarsening dynamics] Evolution of the free energy by the GTransNet-BDF2 scheme with $\overline{u}=0$ (left) and $\overline{u}=0.5$ (right).}
    \label{fig:coarsening_eps002_energy}
\end{figure}

Evolution of the phase variable at $t=5,20$, and $40$ for $\overline{u}\in\{0,0.5\}$ is reported in Figure~\ref{fig:coarsening_eps002_snapshots}. As expected, for $\overline u=0$, the two phases separate into large interconnected regions, whereas for $\overline{u}=0.5$, the minority phase forms isolated circular droplets within the majority phase. To emphasize the importance of the proposed projection, we additionally run the simulation by the GTransNet-BDF2 scheme without the projection, using the same parameters as above, and report the evolution of the absolute mass error for both cases in Figure~\ref{fig:coarsening_eps002_mass}. Without the projection, the mass is clearly not conserved and the error tends to grow over time. With the projection, the mass is conserved throughout the entire simulation for both volume fractions, and the free energy, shown in Figure~\ref{fig:coarsening_eps002_energy}, decays monotonically.

\subsubsection{Cahn-Hilliard equation on irregular domains}
A key advantage of the proposed GTransNet-BDF method is its mesh-free nature, as the hidden-layer neurons are generated from a single domain-covering ball, allowing the method to be applied directly to domains of arbitrary shape. To highlight this flexibility, we simulate the coarsening dynamics governed by the CH equation on three computational domains -- a square, a disk, and an amoeba-shaped region with non-convex boundary -- under homogeneous Neumann boundary conditions. We use the double-well potential $F(u)=\tfrac{1}{4}u^2(u-1)^2$ with wells at $u=0$ and $u=1$ and consider the following initial condition for all three domains~\cite{Dehghan15,Dehghan17,Cao22}:
$$u_0(x,y)=0.5+0.17\cos(\pi x)\cos(2\pi y)+0.2\cos(3\pi x)\cos(\pi y).$$
We fix the common parameters $\varepsilon=0.01$, $D=0.25$, $\kappa=2$, $\delta=0.5$, $N_1=2000$, and $N_2=1000$. Other domain-specific parameters are listed below.
\begin{itemize}
    \item \textbf{Square domain:} $\Omega=(0,1)^2$, $R=0.75$, $\gamma=12$, $K_{\mathrm{int}}=200^2$, $K_{\mathrm{bd}}=800$, $\Delta t = \num{2e-3}$.
    \item \textbf{Circular domain:} $\Omega=\{(x,y): (x-0.5)^2+(y-0.5)^2<0.25\}$, $R=0.6$, $\gamma=12$, $K_{\mathrm{int}}=31428$, $K_{\mathrm{bd}}=800$, $\Delta t = \num{2e-3}$.
    \item \textbf{Amoeba-shaped domain} with boundary
    \begin{align*}
    &\partial\Omega=\left\{(x,y): x=\frac 15\rho(\theta)\cos(\theta)+\frac 25,\;y=\frac 15\rho(\theta)\sin(\theta)+\frac 25,\;0\le \theta<2\pi\right\},\\
    &\rho(\theta)=e^{\sin(\theta)}\sin^2(2\theta)+e^{\cos(\theta)}\cos^2(2\theta),
    \end{align*}
    and $\bm x_c=(0.53,0.48)$, $R=0.65$, $\gamma=13.76$, $K_{\mathrm{int}}=31259$, $K_{\mathrm{bd}}=800$, $\Delta t = \num{5e-4}$.
\end{itemize}

\begin{figure}[!ht]
    \centering
\includegraphics[width=0.3\linewidth]{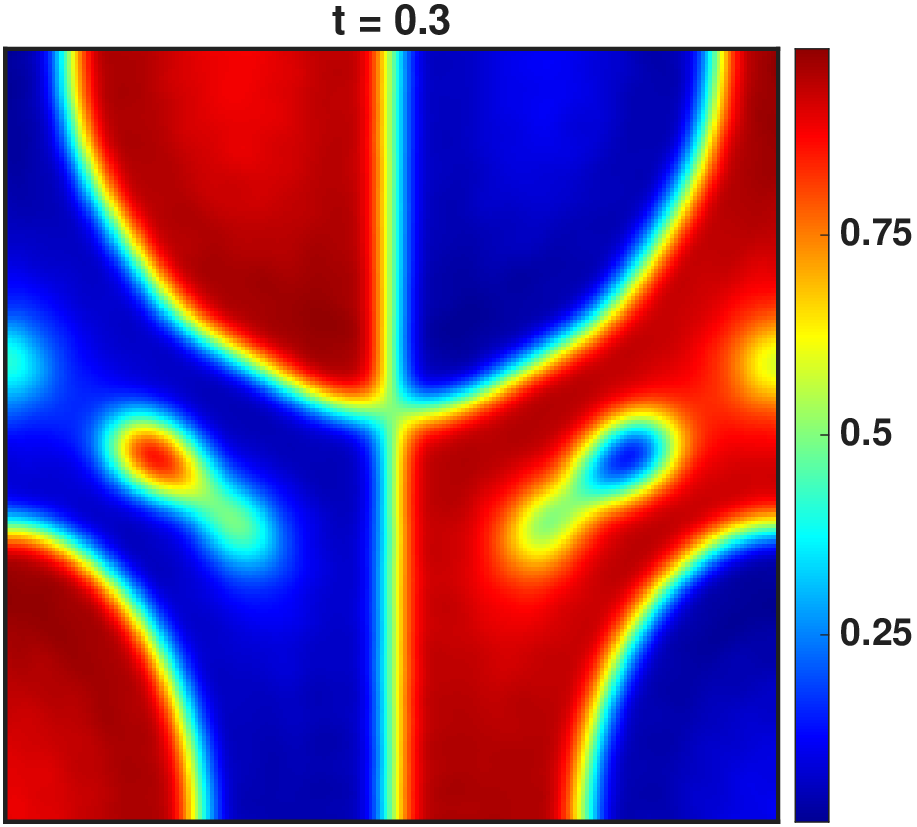}\qquad
\includegraphics[width=0.3\linewidth]{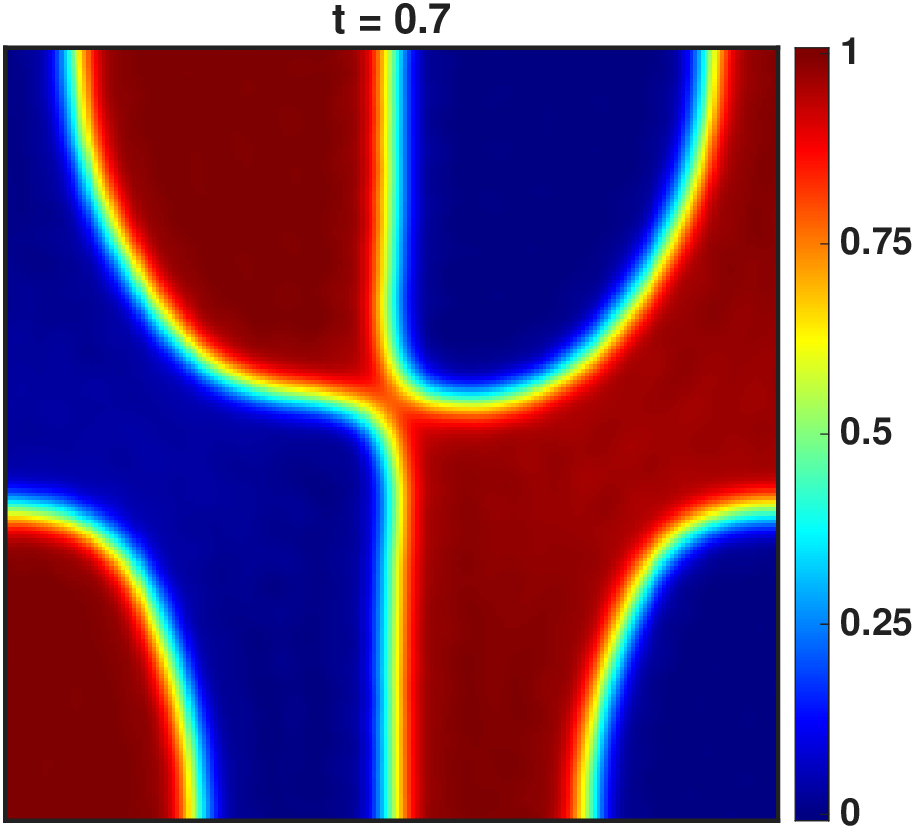}\qquad
\includegraphics[width=0.3\linewidth]{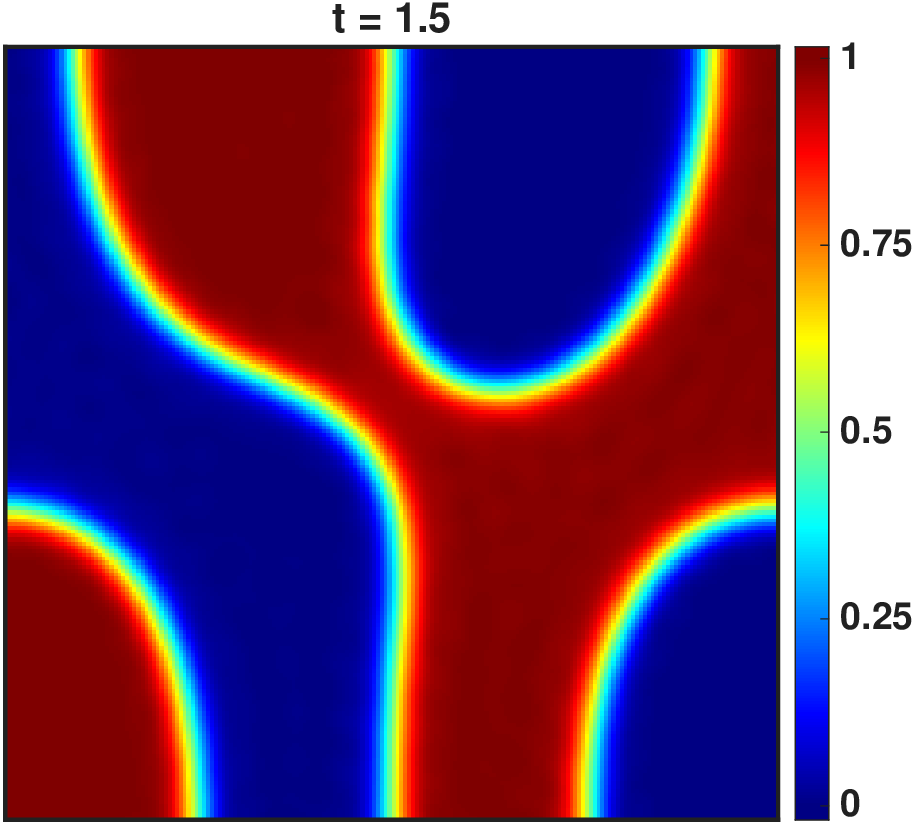}

\vspace{.5cm}

\includegraphics[width=0.3\linewidth]{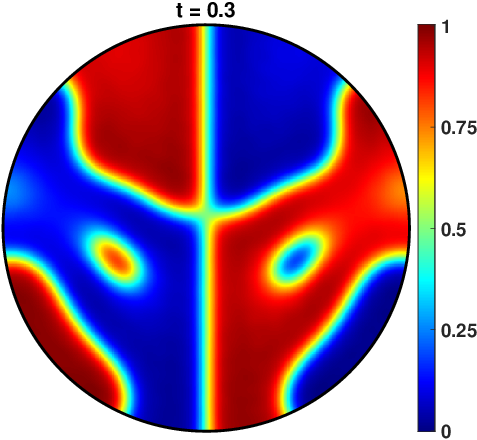}\qquad
\includegraphics[width=0.3\linewidth]{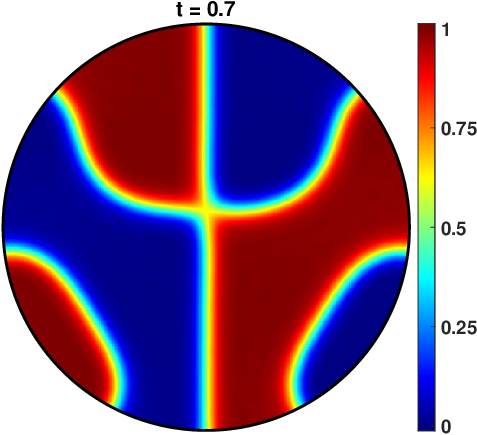}\qquad
\includegraphics[width=0.3\linewidth]{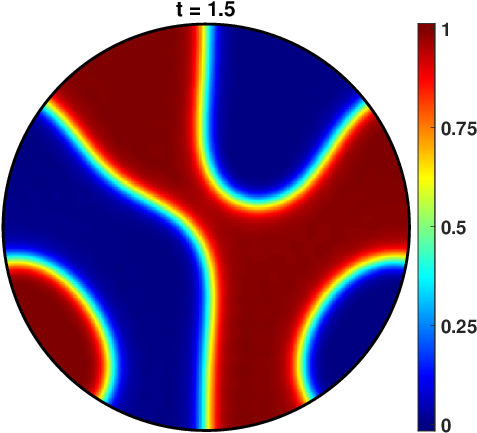}

\vspace{.5cm}

\includegraphics[width=0.3\linewidth]{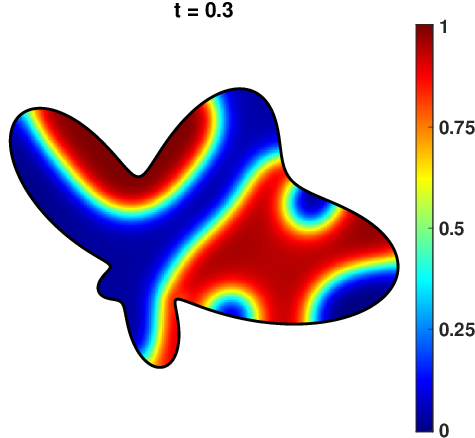}\qquad
\includegraphics[width=0.3\linewidth]{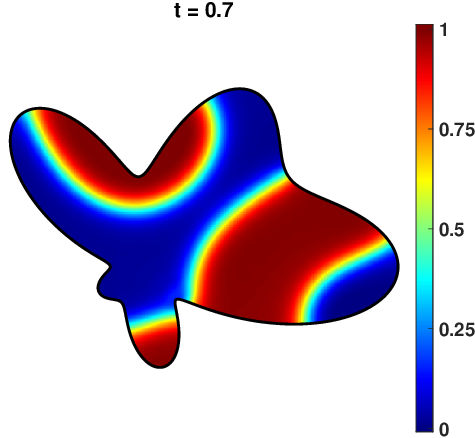}\qquad
\includegraphics[width=0.3\linewidth]{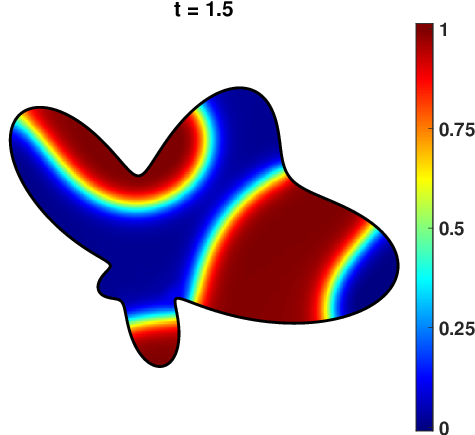}
    \caption{[Regular and irregular domains] Snapshots of the phase variable by the GTransNet-BDF2 scheme at $t=0.3,0.7$, and $1.5$ for the square, circular, and amoeba-shaped domains (top to bottom).}
    \label{fig:irregular_domains_snapshots}
\end{figure}

\begin{figure}[!ht]
    \centering
\includegraphics[width=0.32\linewidth]{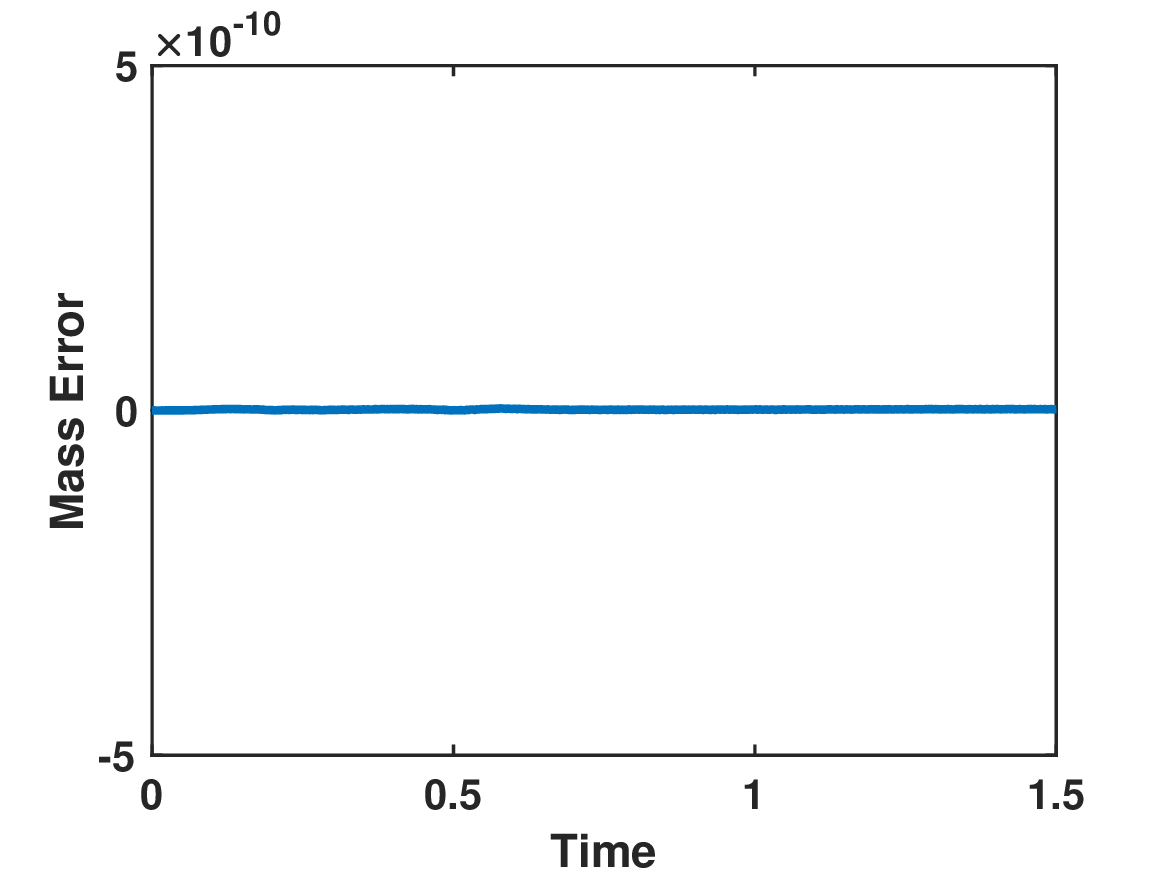}
\includegraphics[width=0.32\linewidth]{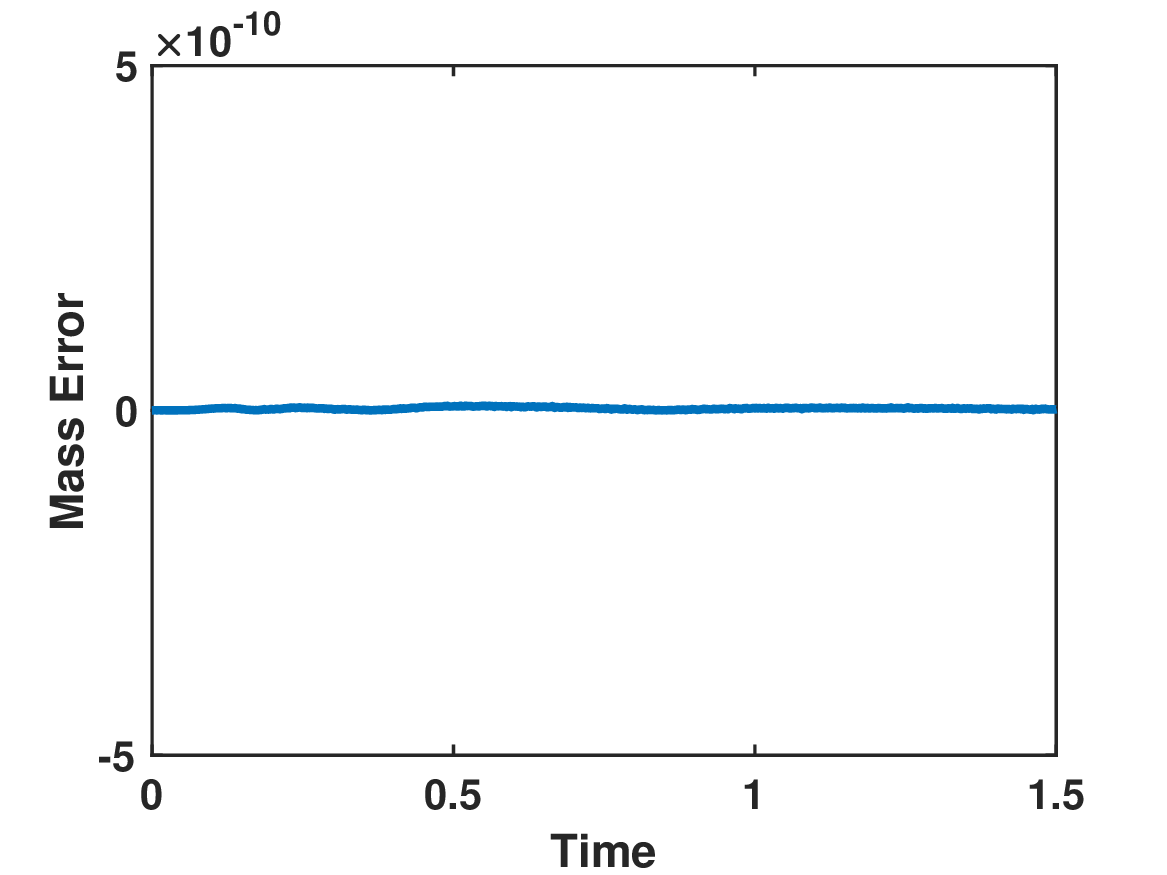}
\includegraphics[width=0.32\linewidth]{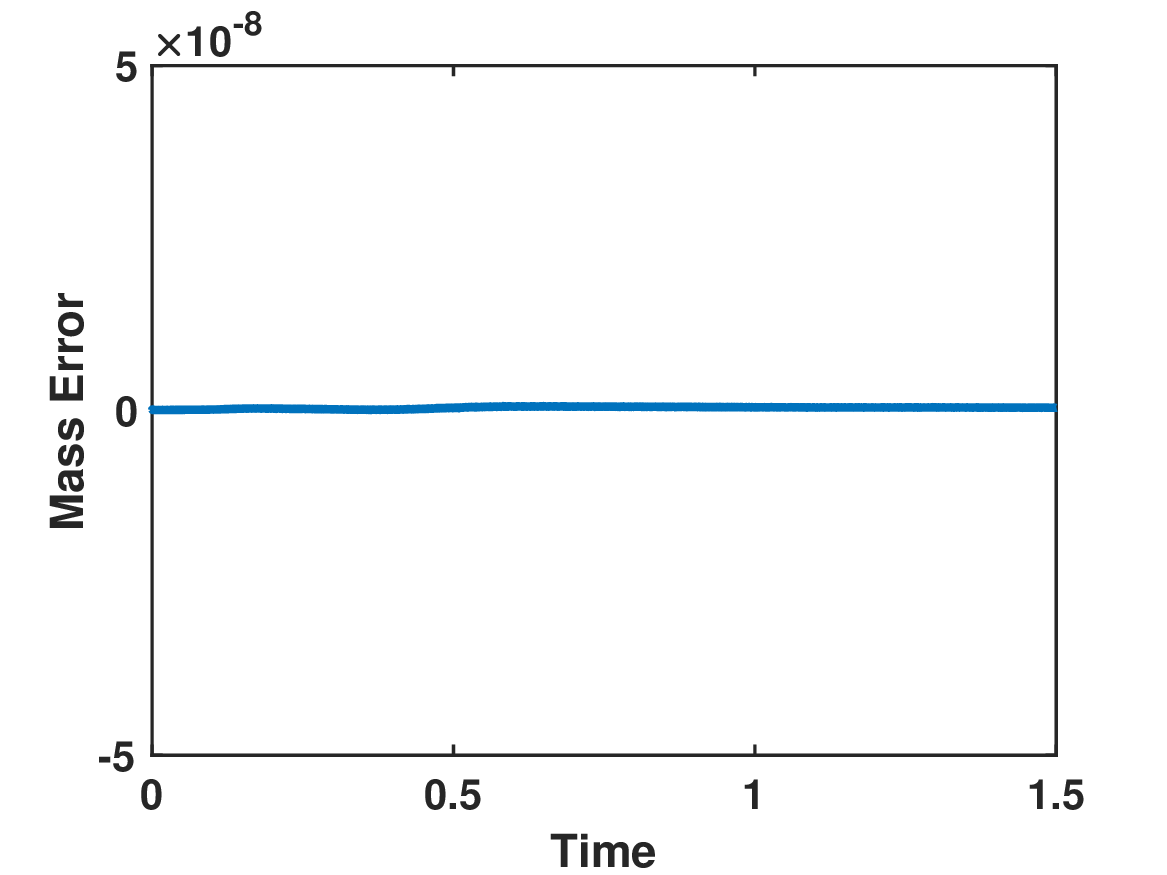}

\includegraphics[width=0.32\linewidth]{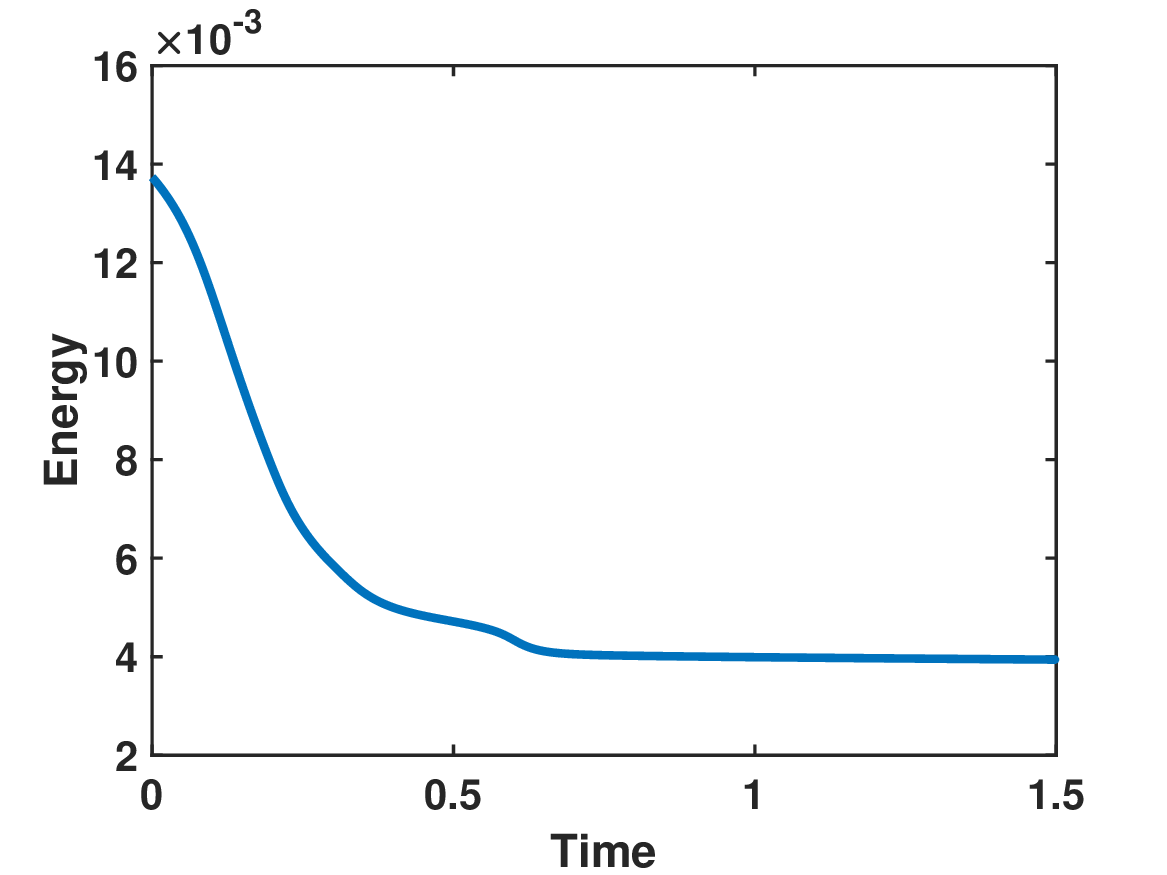}
\includegraphics[width=0.32\linewidth]{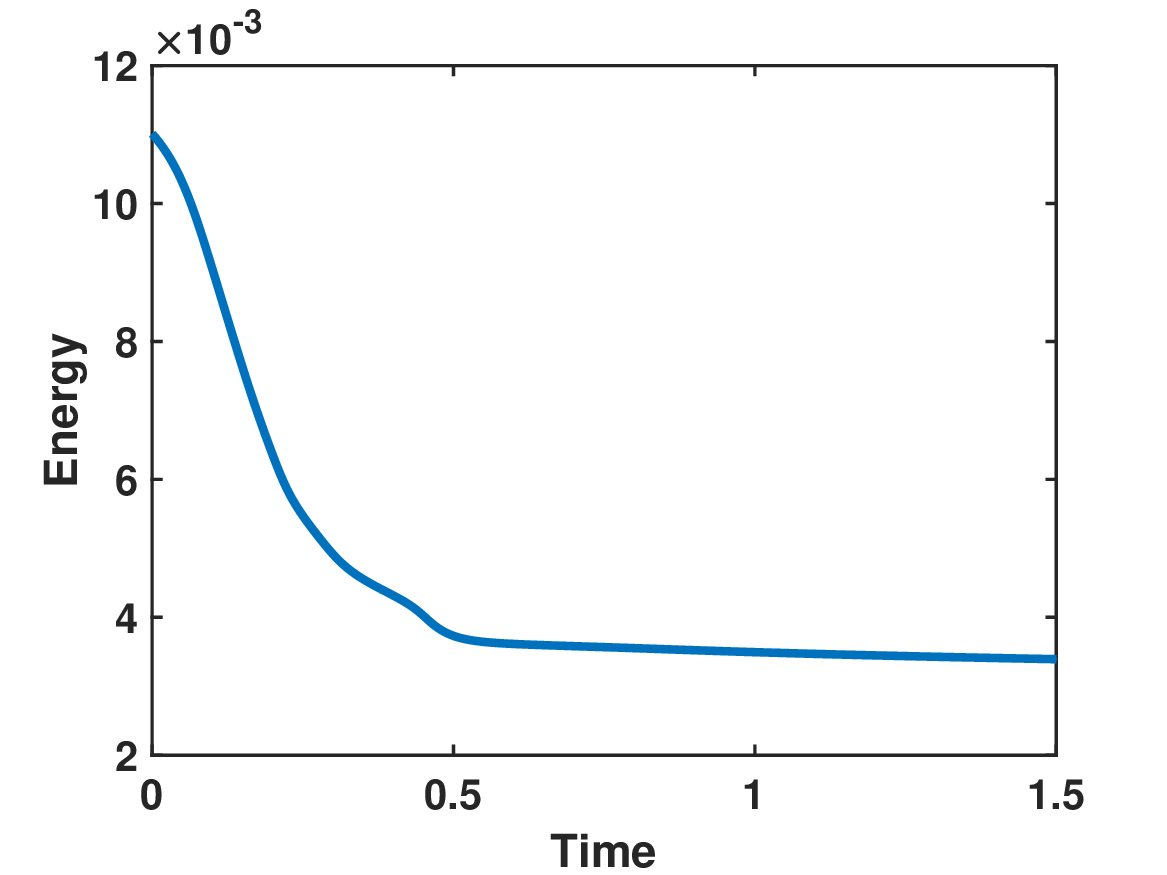}
\includegraphics[width=0.32\linewidth]{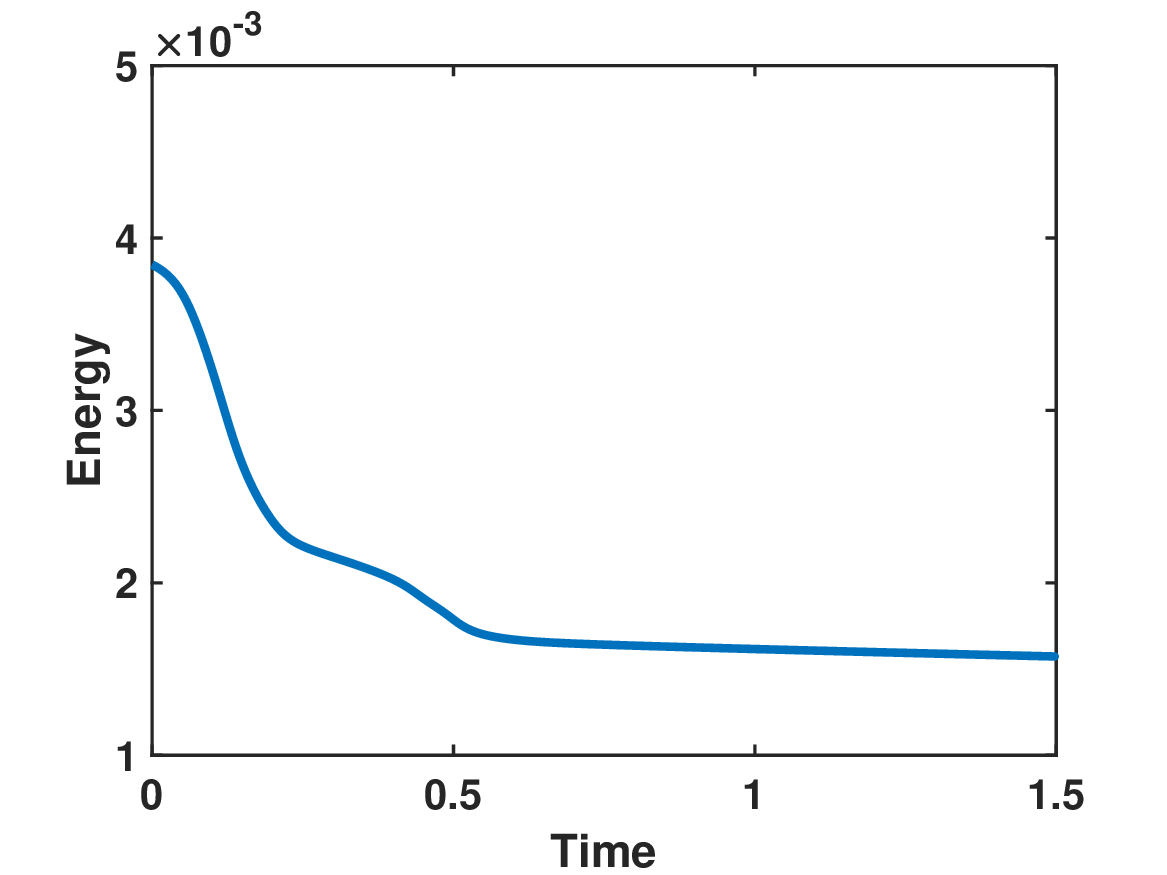}
    \caption{[Regular and irregular domains] Evolution of the absolute mass error (top) and free energy (bottom) by the GTransNet-BDF2 scheme for the square, circular, and amoeba-shaped domains (from left to right).}
    \label{fig:coarsening_eps001_mass_energy}
\end{figure}

Figure~\ref{fig:irregular_domains_snapshots} illustrates the phase variable at $t=0.3$, $0.7$, and $1.5$. In all three cases, the GTransNet-BDF2 scheme captures the early-stage spinodal decomposition and subsequent coarsening, producing smooth and well-resolved interfaces that conform to the domain geometry. The results for the square and circular domains are in strong agreement with those of the radial basis function (RBF) methods in~\cite{Dehghan15,Dehghan17}. For the amoeba-shaped domain, the GTransNet-BDF2 solution remains much closer to the physical range $[0,1]$ than the one reported in~\cite{Cao22}. Figure~\ref{fig:coarsening_eps001_mass_energy} further confirms mass conservation and energy dissipation of the numerical solutions across all three domains.

\subsection{Three-dimensional Cahn-Hilliard equation}\label{subsec:num_3D}

\subsubsection{Convergence test}
We next examine the temporal convergence of the GTransNet-BDF2 scheme in $\Omega=(0,1)^3$ under homogeneous Neumann boundary conditions, using a manufactured solution
\begin{align*}
    u(x,y,z,t)=\cos(\pi x)\cos(\pi y)\cos(\pi z)\sin(t),
\end{align*}
We set $T=1$, $D=0.1$, $\kappa=0$, $\delta=0.5$, $R=0.9$, $K_{\mathrm{int}}=30^3$, $K_{\mathrm{bd}}=5402$, and $\varepsilon\in\{0.2,0.1\}$, with the number of neurons in each hidden layer and the shape parameter listed in Table~\ref{table:params_3d} for each $\varepsilon$. The test data consist of $8K_{\mathrm{int}}$ uniformly distributed random points in $\Omega$.

\begin{table}[!ht]
    \centering
    \small
    \setlength{\extrarowheight}{3pt}
\subcaptionbox{$L=2$ hidden layers}[.49\linewidth]{
    \begin{tabular}{|c|c|c|c|}
    \hline
    $\varepsilon$ & $N_1$ & $N_2$ & $\gamma$ \\
    \hline
    $0.2$ & $2500$ & $1500$ & $1$ \\
    $0.1$ & $3000$ & $1500$ & $3$ \\
    \hline
    \end{tabular}
}\hspace{.05cm}
\subcaptionbox{$L=3$ hidden layers}[.49\linewidth]{
    \begin{tabular}{|c|c|c|c|c|}
    \hline
    $\varepsilon$ & $N_1$ & $N_2$ & $N_3$ & $\gamma$ \\
    \hline
    $0.2$ & $1500$ & $1500$ & $1500$ & $1$\\
    $0.1$ & $1500$ & $1500$ & $1500$ & $3$ \\
    \hline
    \end{tabular}
}
    \caption{[3D convergence test] Parameter settings for $\varepsilon\in\{0.2,0.1\}$.}
    \label{table:params_3d}
\end{table}

Table~\ref{table:convergence_3d} reports the $L^{\infty}$ errors of the phase variable at the final time by the GTransNet-BDF2 scheme with $L=2$ and $L=3$ hidden layers for $\varepsilon\in\{0.2,0.1\}$. As in the two-dimensional case, we observe second-order convergence of the scheme, with comparable errors for $L=2$ and $L=3$.

\begin{table}[!ht]
    \centering
    \small
    \setlength{\extrarowheight}{3pt}
\subcaptionbox{$\varepsilon=0.2$}[.49\linewidth]{
\resizebox{.49\textwidth}{!}{
\begin{tabular}{|c|c|c|}
\hline
\multirow{2}{*}{$\Delta t$} & GTransNet-BDF2 & GTransNet-BDF2 \\
 & $L=2$ hidden layers & $L=3$ hidden layers\\ \hline
$1/10$  & 8.58e-03         & 8.58e-03 \\
$1/20$  & 1.99e-03 [2.11]  & 1.99e-03 [2.11]\\
$1/40$  & 4.90e-04 [2.02]  & 4.89e-04 [2.02]\\
$1/80$  & 1.23e-04 [1.99]  & 1.22e-04 [2.00]\\
$1/160$ & 3.09e-05 [1.99]  & 3.11e-05 [1.97]\\
$1/320$ & 7.48e-06 [2.05]  & 8.47e-06 [1.88]\\
\hline
\end{tabular}
}
}\hspace{.05cm}
\subcaptionbox{$\varepsilon=0.1$}[.49\linewidth]{
\resizebox{.49\textwidth}{!}{
\begin{tabular}{|c|c|c|}
\hline
\multirow{2}{*}{$\Delta t$} & GTransNet-BDF2 & GTransNet-BDF2 \\
 & $L=2$ hidden layers & $L=3$ hidden layers \\ \hline
$1/10$  & 8.19e-02         & 8.19e-02 \\
$1/20$  & 2.33e-02 [1.81]  & 2.33e-02 [1.81] \\
$1/40$  & 6.18e-03 [1.91]  & 6.18e-03 [1.91] \\
$1/80$  & 1.58e-03 [1.97]  & 1.58e-03 [1.97] \\
$1/160$ & 3.98e-04 [1.99]  & 3.97e-04 [1.99] \\
$1/320$ & 1.05e-04 [1.92]  & 9.66e-05 [2.04] \\
\hline
\end{tabular}
}
}
    \caption{[3D convergence test] $L^{\infty}$ errors of the phase variable at the final time for $\varepsilon\in\{0.2,0.1\}$.}
    \label{table:convergence_3d}
\end{table}

\subsubsection{Coarsening dynamics}\label{subsubsec:coarsening_3D}
We proceed to consider the three-dimensional coarsening dynamics in $\Omega=(0,1)^3$ under homogeneous Neumann boundary conditions, a natural extension of the two-dimensional case. The initial condition is again a small random perturbation of a uniform state
$$u_0(x,y,z)=\overline{u}+0.05\,\mathrm{rand}(x,y,z),$$
with $\mathrm{rand}(x,y,z)$ uniformly distributed in $[-1,1]$. We take $\varepsilon=0.04$, $\overline{u}=0$, $D=0.05$, $\kappa=2$, $\delta=0.5$, $R=0.9$, $K_{\mathrm{int}}=50^3$, and $K_{\mathrm{bd}}=15002$. The simulation is run until $T=20$ using the GTransNet-BDF2 scheme with $N_1=4000$, $N_2=2000$, $\gamma=7$, and $\Delta t=\num{5e-3}$.

\begin{figure}[!ht]
    \centering
\includegraphics[width=0.46\linewidth]{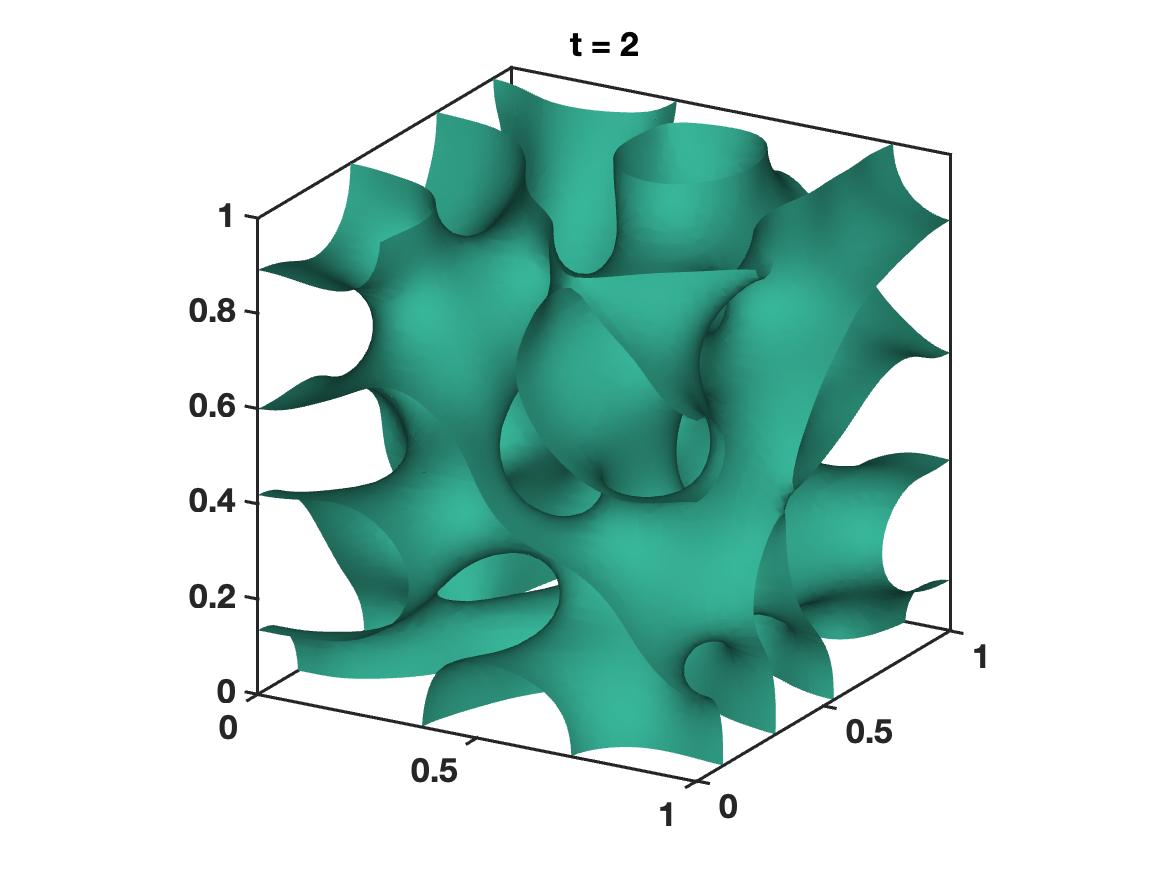}\qquad
\includegraphics[width=0.46\linewidth]{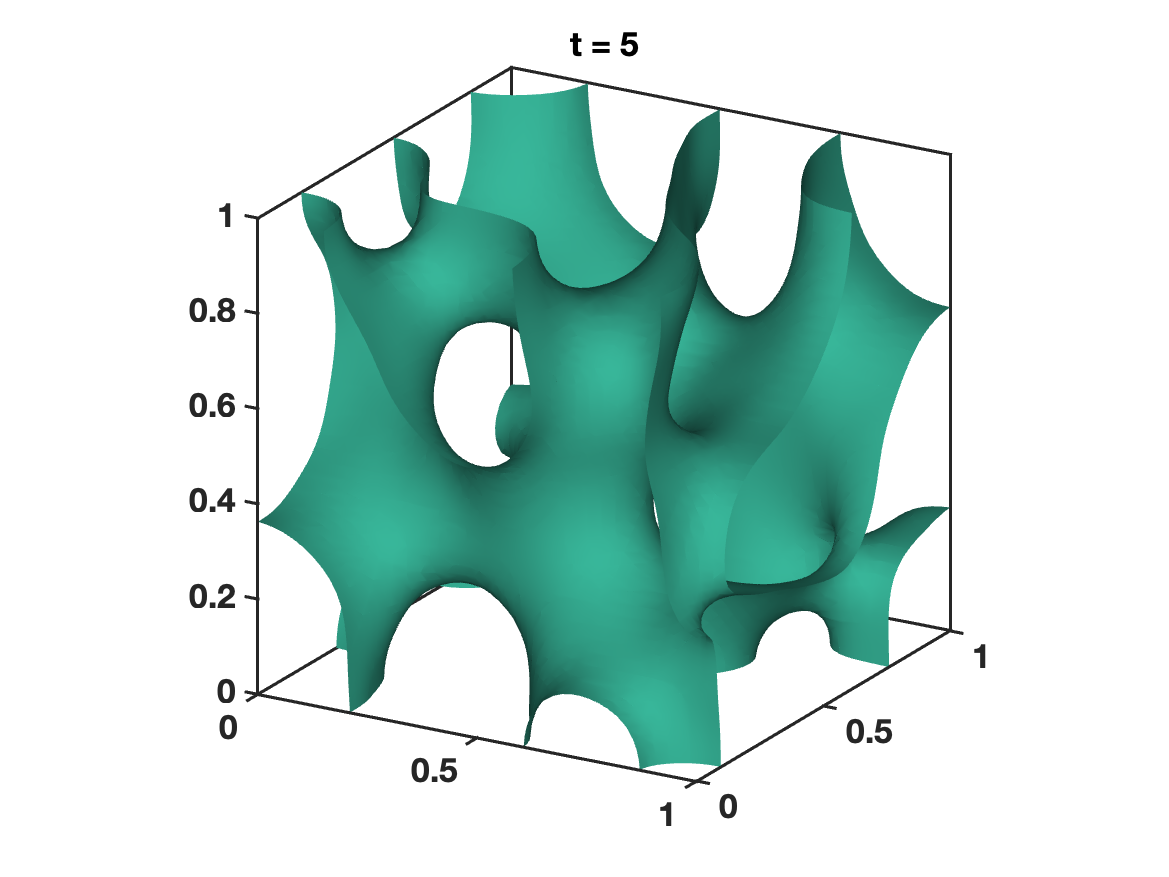}

\vspace{0.2cm}

\includegraphics[width=0.46\linewidth]{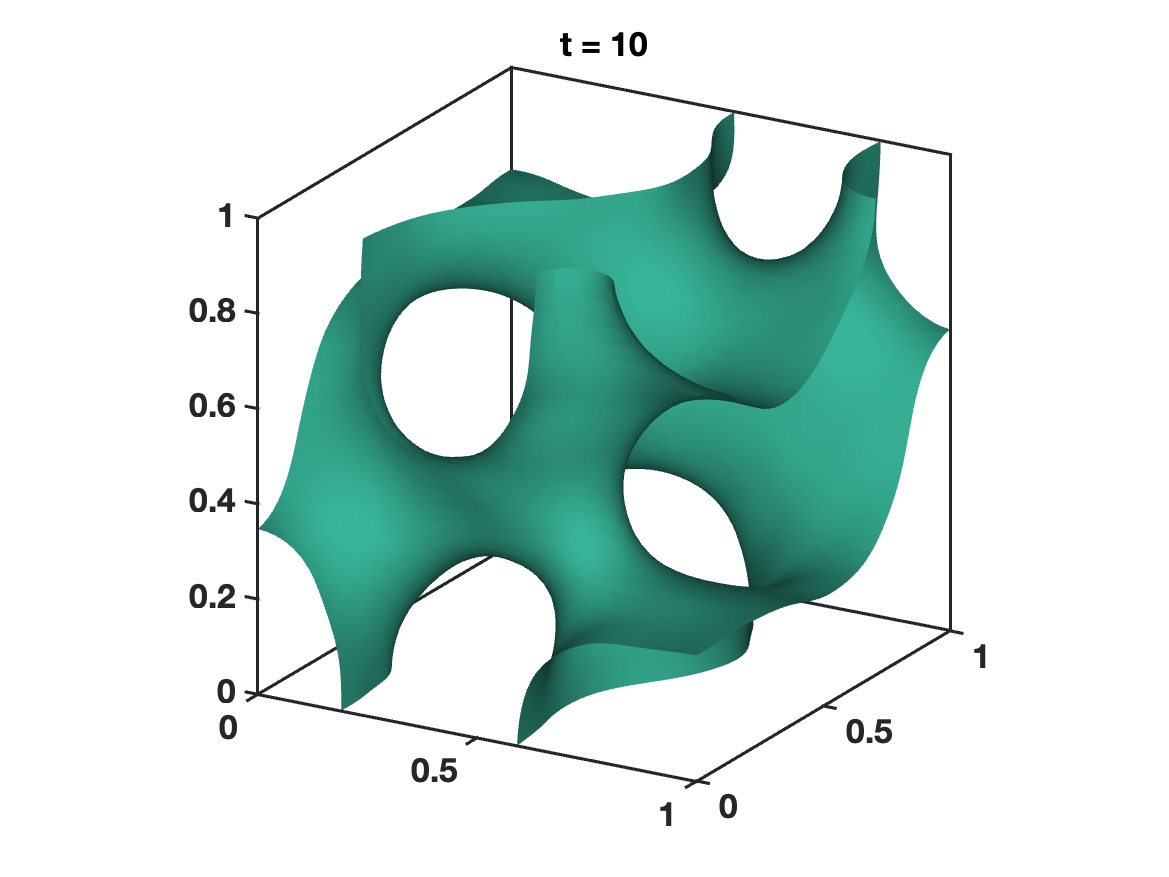}\qquad
\includegraphics[width=0.46\linewidth]{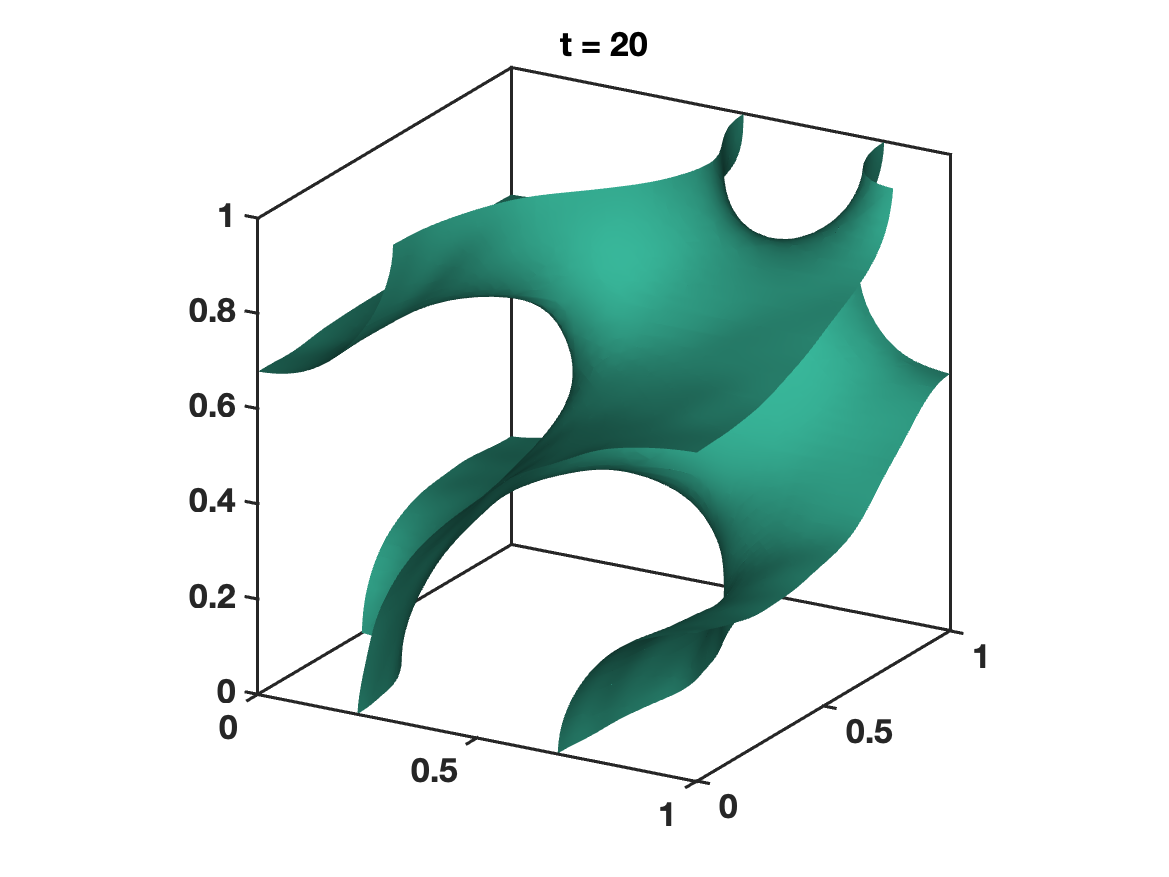}
    \caption{[3D coarsening dynamics] Zero-isosurfaces of the phase variable by the GTransNet-BDF2 scheme at $t=2,5,10$, and $20$.}
    \label{fig:coarsening3D_eps004_snapshots}
\end{figure}

\begin{figure}[!ht]
    \centering
\includegraphics[width=0.4\linewidth]{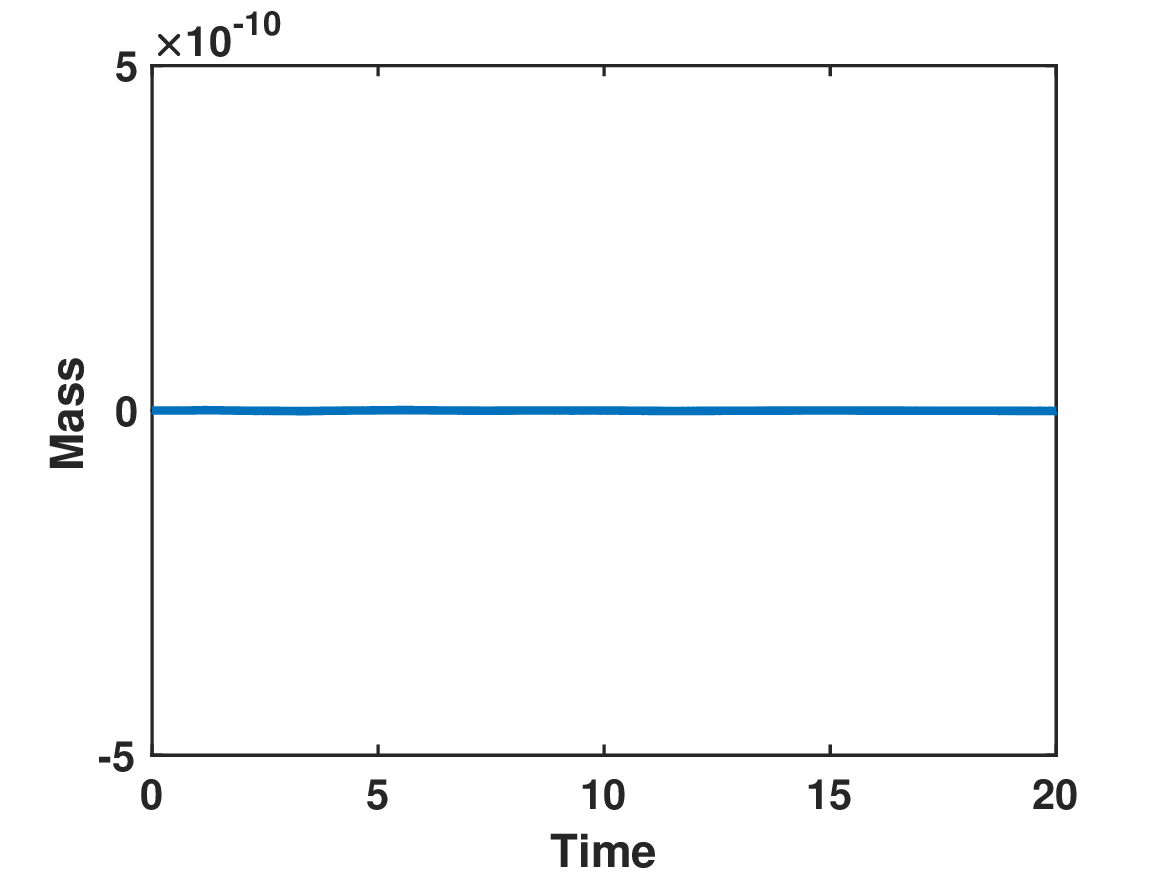}
\includegraphics[width=0.4\linewidth]{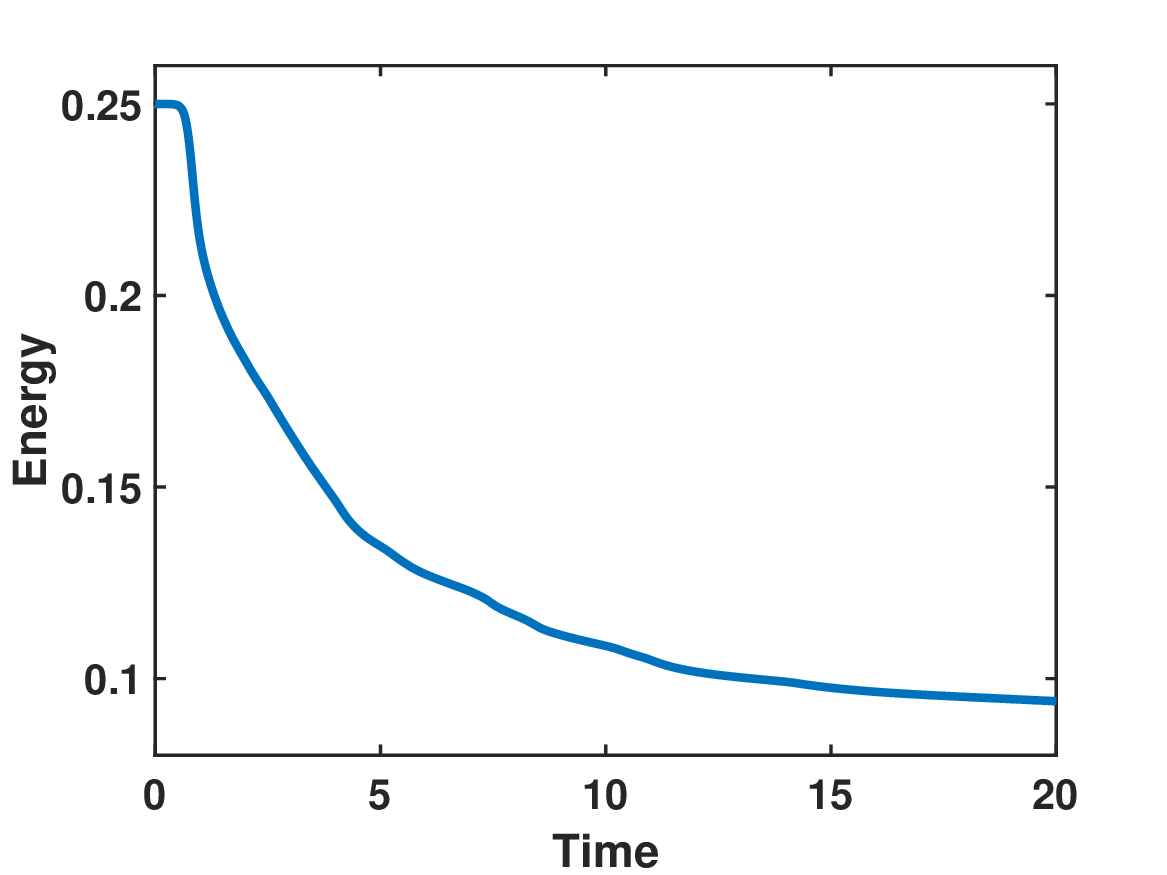}
    \caption{[3D coarsening dynamics] Evolution of the absolute mass error and free energy by the GTransNet-BDF2 scheme.}
    \label{fig:coarsening3D_eps004_mass_energy}
\end{figure}

Isosurface snapshots of the phase variable (the zero level set $\{u=0\}$) at different times by the GTransNet-BDF2 scheme are shown in Figure~\ref{fig:coarsening3D_eps004_snapshots}. Starting from the random initial configuration, the interface coarsens over time, with the fine, highly connected structure at $t=2$ progressively simplifying into a smoother surface with fewer features by $t=20$. Figure~\ref{fig:coarsening3D_eps004_mass_energy} reports the corresponding evolution of the absolute mass error and free energy, confirming the preservation of the two intrinsic properties.

\subsection{Cahn-Hilliard equation with degenerate mobility}

Finally, we extend the proposed GTransNet-BDF framework to the CH equation with variable mobility $M(u)$, for which the corresponding first- and second-order schemes are derived in~\ref{sec:appendix_varmob}. We consider the symmetric degenerate mobility $M(u)=|1-u^2|,$ which vanishes at the pure phases $u=\pm 1$ and attains its maximum $\Gamma=\max_{u\in[-1,1]} M(u)=1$ at $u=0$. The degeneracy suppresses diffusion in the bulk regions and confines mass transport to the thin interfacial layers, so that coarsening proceeds mainly through surface diffusion and the late-stage dynamics are considerably slower than in the constant-mobility case (i.e., $M(u)\equiv 1$).

We solve the CH equation in $\Omega=(0,1)^2$ under homogeneous Neumann boundary conditions, with the following initial data~\cite{Ju15}:
$$u_0(x,y) = \overline{u}+0.2\,\mathrm{rand}(x,y),$$
where $\mathrm{rand}(x,y)$ is uniformly distributed in $[-1,1]$ and $\overline u\in\{0,0.5\}$. We take $\varepsilon=0.02$, $\kappa=2$, $\delta=0.5$, $R=0.75$, $K_{\mathrm{int}}=250^2$, $K_{\mathrm{bd}}=1000$, and $\Delta t=\num{5e-4}$, and run the simulation until $T=5$ using the GTransNet-BDF2 scheme with $N_1=3000$, $N_2=1500$, $\gamma=14$ for $\overline u=0$, and $N_1=4000$, $N_2=2000$, $\gamma=16$ for $\overline u=0.5$.

\begin{figure}[!ht]
    \centering
\includegraphics[width=0.3\linewidth]{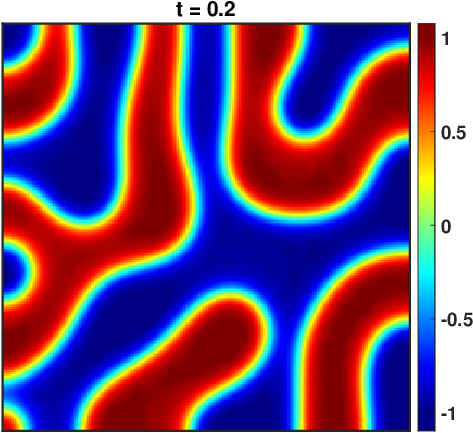}\qquad
\includegraphics[width=0.3\linewidth]{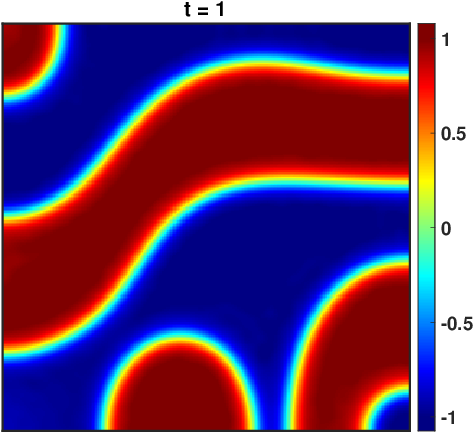}\qquad
\includegraphics[width=0.3\linewidth]{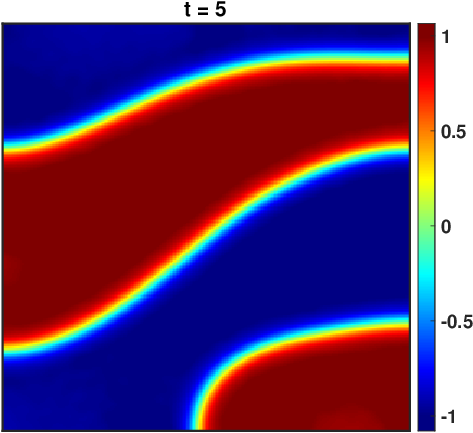}

\vspace{.4cm}    
    
\includegraphics[width=0.3\linewidth]{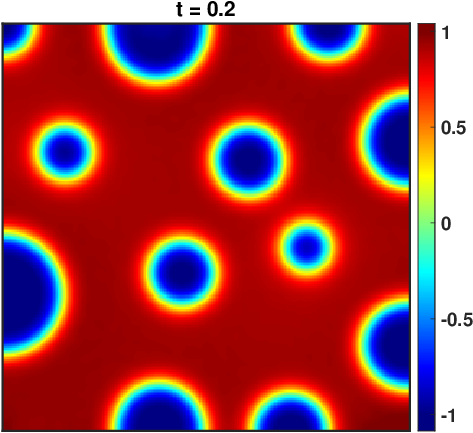}\qquad
\includegraphics[width=0.3\linewidth]{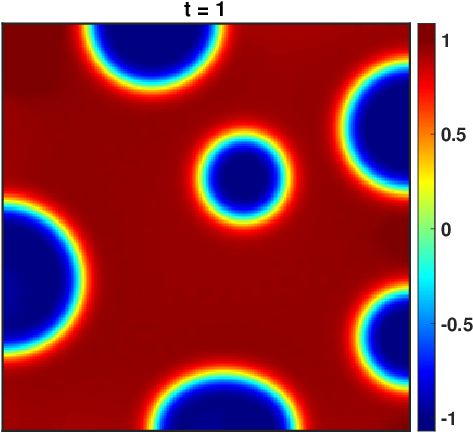}\qquad
\includegraphics[width=0.3\linewidth]{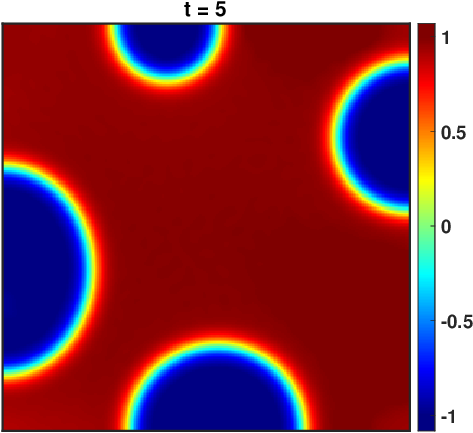}
    \caption{[Degenerate mobility] Snapshots of the phase variable by the GTransNet-BDF2 scheme with $M(u)=|1-u^2|$ for $\overline{u}=0$ (top) and $\overline{u}=0.5$ (bottom) at $t=0.2, 1$, and $5$.}
    \label{fig:mobility_eps002_snapshots}
\end{figure}

\begin{figure}[!ht]
    \centering
\includegraphics[width=0.35\linewidth]{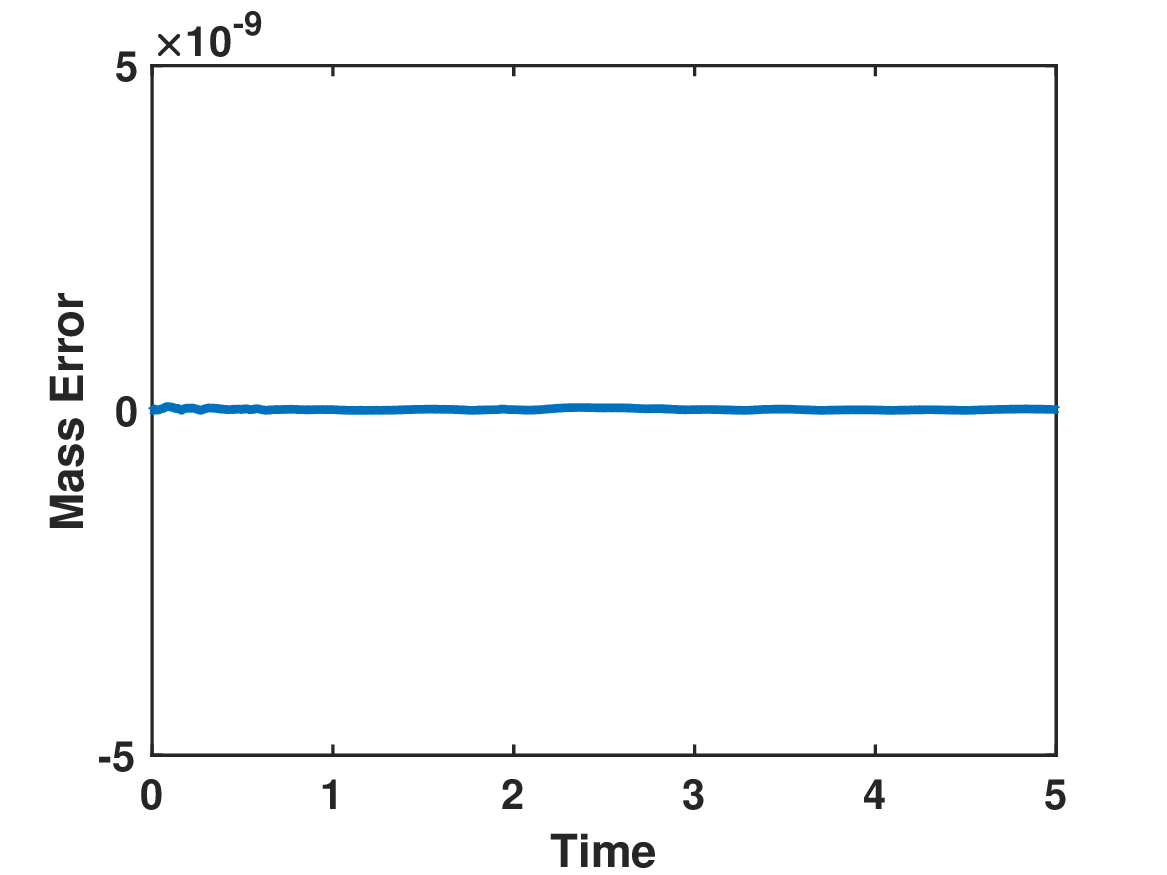}
\includegraphics[width=0.35\linewidth]{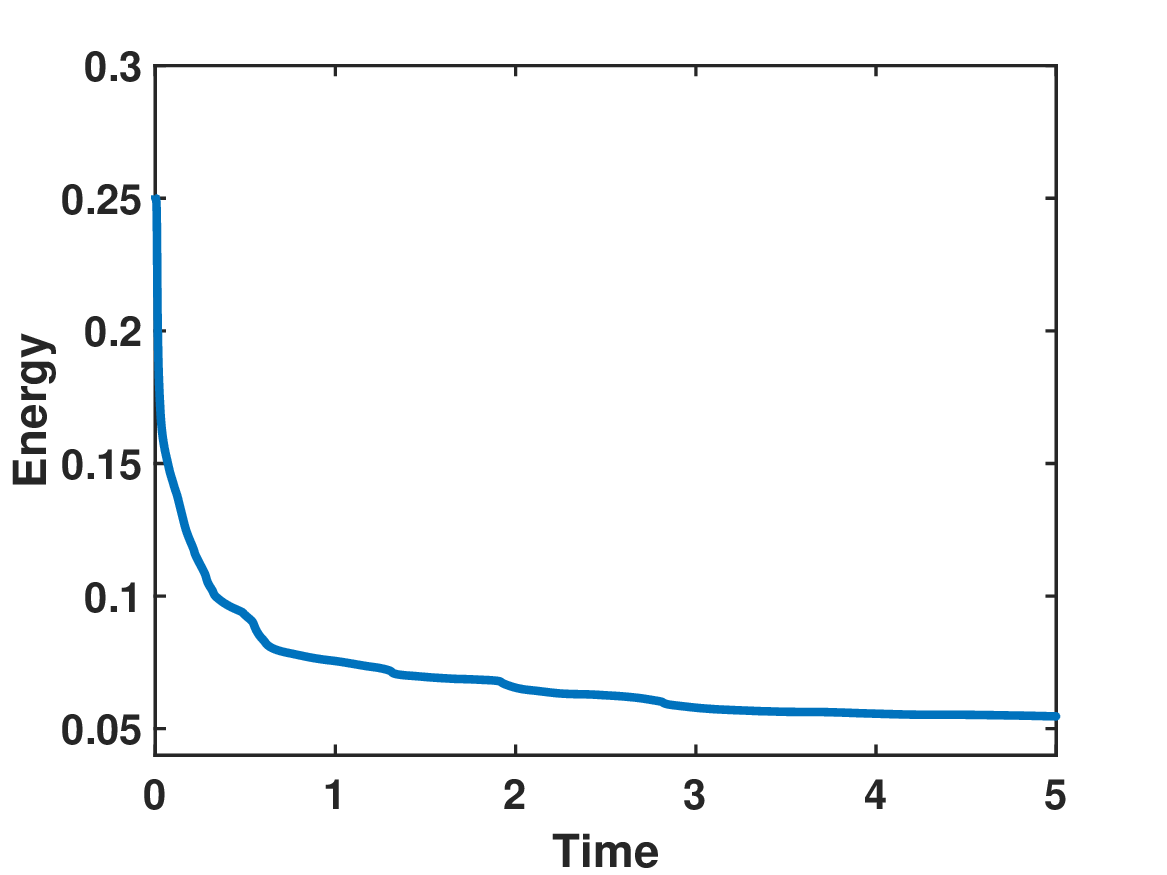}    
    
\includegraphics[width=0.35\linewidth]{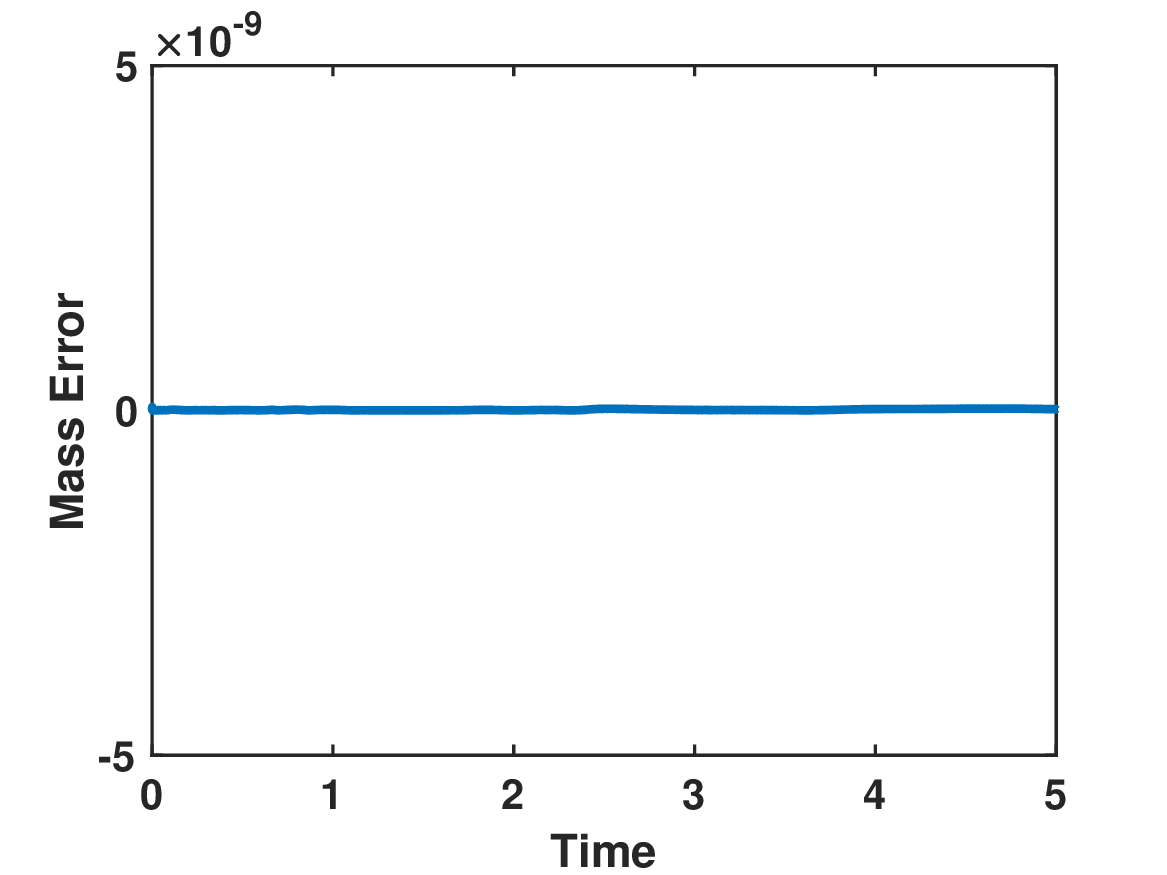}
\includegraphics[width=0.35\linewidth]{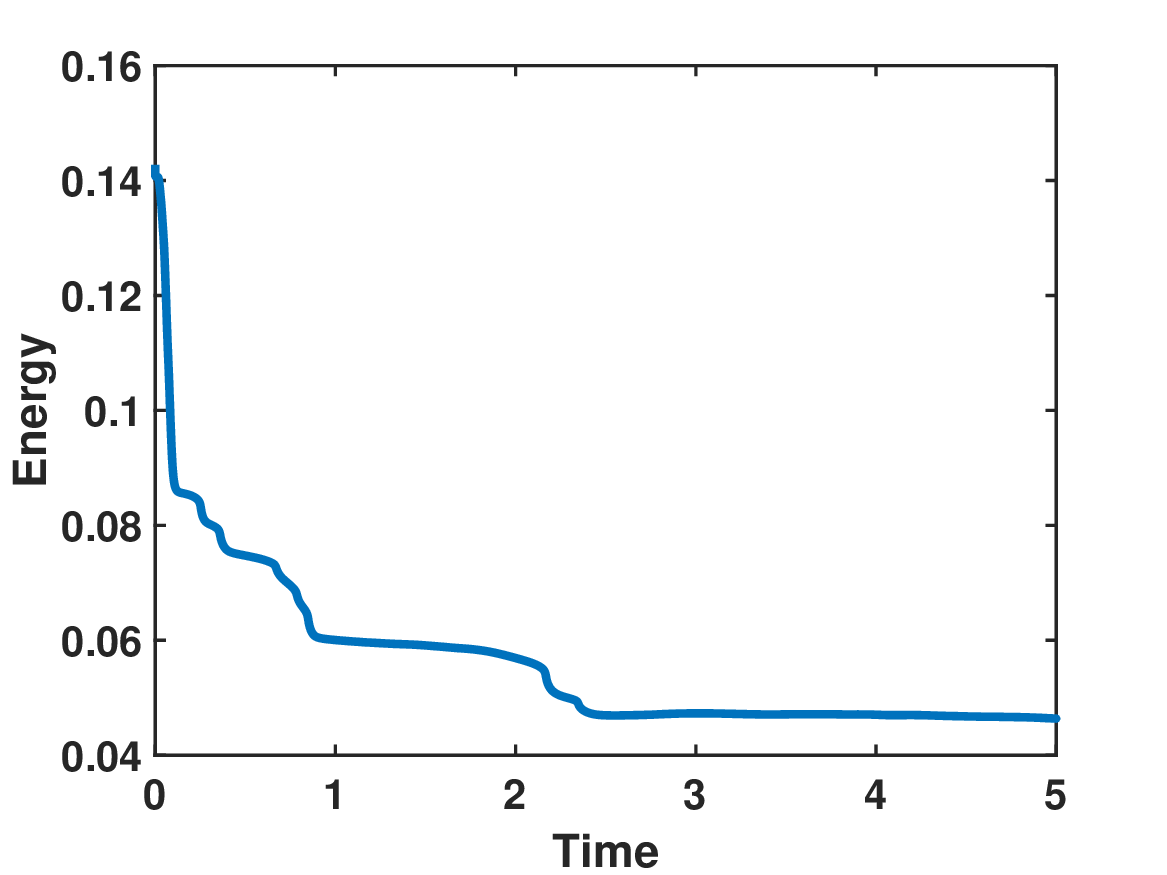}
    \caption{[Degenerate mobility] Evolution of the absolute mass error and free energy by the GTransNet-BDF2 scheme with $M(u)=|1-u^2|$ for $\overline{u}=0$ (top) and $\overline{u}=0.5$ (bottom).}
    \label{fig:mobility_eps002_mass_energy}
\end{figure}

The phase variable for $\overline u\in\{0,0.5\}$ at $t=0.2, 1$, and $5$ is plotted in Figure~\ref{fig:mobility_eps002_snapshots}. For $\overline u=0$, the phases form elongated bands that coarsen into a few wide stripes, while for $\overline u=0.5$, the minority phase appears as circular droplets that grow and merge into larger ones over time. Figure~\ref{fig:mobility_eps002_mass_energy} verifies mass conservation and energy dissipation of the numerical solution, demonstrating that the framework extends robustly to the degenerate-mobility regime.

\section{Conclusions}\label{sec:conclusion}

In this work, we proposed the first- and second-order GTransNet-BDF schemes coupled with the mass-conserving projection for solving the CH equation. The method combines GTransNet architecture for spatial approximation with stabilized BDF for time integration; for the former, the first hidden layer is predetermined as in the original TransNet with a symmetric bias distribution, followed by a variance-controlled weight sampling strategy for the subsequent hidden layers. The schemes preserve two important properties of the CH equation at the time-discrete level, namely mass conservation and energy stability. By employing the collocation-based method, a least-squares system for the unknown output-layer weights is solved at each time step, which can be implemented efficiently using a precomputed QR factorization. To mitigate the effect of least-squares error, we introduced the mass-conserving projection that corrects the least-squares solution along a single direction with minimal perturbation so that the discrete mass is conserved up to a satisfactory accuracy. Numerical experiments in two and three dimensions were  presented to confirm the theoretical findings and illustrate the performance of GTransNet-BDF methods for the CH equation with constant and degenerate mobility.


Several natural directions remain for future research on the GTransNet-BDF framework. First, the collocation-based implementation relies on a large number of collocation points, so developing GTransNet-BDF variants that achieve comparable accuracy with substantially fewer points would improve efficiency. Second, an explicit formula for the optimal shape parameter $\gamma$ in the GTransNet activation is not yet available; our numerical experiments suggest it depends on the spatial dimension $d$, the number of neurons in the first hidden layer $N_1$, the interfacial thickness $\varepsilon$, and is sensitive to the time step size $\Delta t$, especially when $\Delta t\ll 1$. Establishing a precise relationship between $\gamma$ and these parameters is therefore an important topic for future work. The extension of the current framework to other complicated time-dependent PDEs, such as the incompressible Navier-Stokes equations~\cite{Doan25a} and the coupled Cahn-Hilliard-Navier-Stokes system~\cite{Doan25b}, as well as its combination with domain decomposition techniques~\cite{Lu25} for improved computational efficiency on large-scale problems, requires further investigation.



\section*{CRediT authorship contribution statement}
{\bf Cao-Kha Doan}: Methodology, Software, Validation, Writing-original draft;
{\bf Thi-Thao-Phuong Hoang}: Conceptualization, Methodology, Project administration, Writing-reviewing and editing;
{\bf Lili Ju}: Conceptualization, Methodology, Writing-reviewing and editing;
{\bf Shuting Wang}: Validation, Writing-reviewing and editing.

\section*{Declaration of Interests}
The authors have not disclosed any competing interests.

\section*{Data Availability}
No data was used for the research described in the article.

\section*{Acknowledgements}
T.-T.-P. Hoang's work is partially supported by U.S. National Science Foundation under grant number DMS-2041884.  L. Ju's work is partially supported by U.S. National Science Foundation under grant number DMS-2409634 and U.S. Department of Energy under grant number DE-SC0025527.


\appendix

\section{GTransNet-BDF schemes for the Cahn-Hilliard equation with variable mobility}\label{sec:appendix_varmob}

We extend the GTransNet-BDF framework to the CH equation with variable mobility, which takes the following form
\begin{align}\label{eq:CHmix_varmob}
\begin{cases}
\dfrac{\partial u}{\partial t}=\nabla\!\cdot\!\bigl(M(u)\,\nabla\mu\bigr), & \text{in }\Omega\times(0,T],\\[4pt]
\mu=-\varepsilon^2\Delta u+f(u), & \text{in }\Omega\times(0,T],
\end{cases}
\end{align}
subject to the same initial and boundary conditions as in Section~\ref{sec:GTransNet_BDF}. We remark that a direct application of the GTransNet-BDF method to~\eqref{eq:CHmix_varmob} leads to solution-dependent coefficients in the least-squares matrix~\eqref{LSQ}. To recover the one-time factorization of the coefficient matrix for efficient implementation, we adopt the mobility splitting strategy introduced in~\cite{Ju15}. Toward that end, let
$\Gamma:=\max_{u}M(u)$ and consider the decomposition
\begin{align*}
M(u)\,\nabla\mu=\Gamma\,\nabla\mu+\bigl(M(u)-\Gamma\bigr)\nabla\mu,
\end{align*}
so that the first equation in~\eqref{eq:CHmix_varmob} can be written equivalently as
\begin{align}\label{eq:CHmix_varmob_split}
\frac{\partial u}{\partial t}=\Gamma\Delta\mu+\nabla\!\cdot\!\Bigl[\bigl(M(u)-\Gamma\bigr)\nabla\mu\Bigr].
\end{align}
The first term on the right-hand side of~\eqref{eq:CHmix_varmob_split} carries a constant coefficient and is treated implicitly; the second term, in which $M(u)-\Gamma\le 0$, is treated explicitly via extrapolation. As in Section~\ref{sec:GTransNet_BDF}, we let $\kappa\ge 0$ be a stabilization constant, set $f_\kappa(u)=f(u)-\kappa u$, and approximate both $u$ and $\mu$ by their GTransNet expansions~\eqref{u_mu_GTransNet} in the basis~\eqref{GTransNet}.

\subsection{GTransNet-BDF schemes}

Applying the backward Euler method to~\eqref{eq:CHmix_varmob}-\eqref{eq:CHmix_varmob_split} with explicit treatment of the nonlinear term and the variable mobility correction term yields the first-order GTransNet-BDF1 scheme
\begin{subequations}\label{GBDF1_varmob}
\begin{align}
\sum_{j=1}^{N_L}\frac{\alpha_j^{n+1}-\alpha_j^{n}}{\Delta t}\phi_j(\bm x)
&=\Gamma\sum_{j=1}^{N_L}\beta_j^{n+1}\Delta\phi_j(\bm x)
   +\nabla\!\cdot\!\Bigl[\bigl(M(u_{\mathrm{NN}}^n(\bm x))-\Gamma\bigr)\nabla\mu_{\mathrm{NN}}^{\star,n}(\bm x)\Bigr],\label{GBDF1a_varmob}\\
\sum_{j=1}^{N_L}\beta_j^{n+1}\phi_j(\bm x)
&=\sum_{j=1}^{N_L}\alpha_j^{n+1}(-\varepsilon^2\Delta+\kappa)\phi_j(\bm x)+f_{\kappa}(u_{\mathrm{NN}}^n(\bm x)),\label{GBDF1b_varmob}
\end{align}
\end{subequations}
where
$\mu_{\mathrm{NN}}^{\star,n}=-\varepsilon^2\Delta u_{\mathrm{NN}}^n+f(u_{\mathrm{NN}}^n)$. Similarly, the corresponding second-order GTransNet-BDF2 scheme is given by
\begin{subequations}\label{GBDF2_varmob}
\begin{align}
\sum_{j=1}^{N_L}\frac{3\alpha_j^{n+1}-4\alpha_j^n+\alpha_j^{n-1}}{2\Delta t}\phi_j(\bm x)
&=\Gamma\sum_{j=1}^{N_L}\beta_j^{n+1}\Delta\phi_j(\bm x)
   +\nabla\!\cdot\!\Bigl[\bigl(M(\tilde u_{\mathrm{NN}}^{n+1}(\bm x))-\Gamma\bigr)\nabla\tilde\mu_{\mathrm{NN}}^{n+1}(\bm x)\Bigr],\label{GBDF2a_varmob}\\
\sum_{j=1}^{N_L}\beta_j^{n+1}\phi_j(\bm x)
&=\sum_{j=1}^{N_L}\alpha_j^{n+1}(-\varepsilon^2\Delta+\kappa)\phi_j(\bm x)+f_{\kappa}(\tilde u_{\mathrm{NN}}^{n+1}(\bm x)),\label{GBDF2b_varmob}
\end{align}
\end{subequations}
where  $\tilde u_{\mathrm{NN}}^{n+1}=2u_{\mathrm{NN}}^n-u_{\mathrm{NN}}^{n-1}$ and
$\tilde\mu_{\mathrm{NN}}^{n+1}=-\varepsilon^2\Delta\tilde u_{\mathrm{NN}}^{n+1}+f(\tilde u_{\mathrm{NN}}^{n+1}).$
For the initialization step, $\{\alpha_j^1\}_{j=1}^{N_L}$ is computed by the GTransNet-BDF1 scheme~\eqref{GBDF1_varmob}.

\subsection{Implementation}

The least-squares systems for the schemes~\eqref{GBDF1_varmob} and~\eqref{GBDF2_varmob} retain the structure of~\eqref{LSQ} with two modifications: the diffusion coefficient $D$ in the implicit term is replaced by $\Gamma$, and the first block of the right-hand side vector includes an explicit variable-mobility contribution. For the GTransNet-BDF1 scheme~\eqref{GBDF1_varmob}, we solve for $\bm c^{n+1}\in \mathbb R^{2N_L}$ the following least-squares system
\begin{align}\label{LSQ_varmob_BDF1}
\begin{pmatrix}
\bm\Phi & -\Gamma\Delta t\,\bm\Phi_{\Delta}\\[2pt]
\varepsilon^2\bm\Phi_{\Delta}-\kappa\bm\Phi & \bm\Phi\\
\bm\Phi_{\mathrm{bd}} & \bm 0\\
\bm 0 & \bm\Phi_{\mathrm{bd}}
\end{pmatrix}\bm c^{n+1}
=\begin{pmatrix}
\bm\Phi\bm\alpha^n+\Delta t\,\bm g_{\mathrm{var}}^{n}\\[2pt]
f_\kappa(\bm\Phi\bm\alpha^n)\\
\bm 0\\
\bm 0
\end{pmatrix},
\end{align}
where $\bm g_{\mathrm{var}}^{n}\in\mathbb R^{K_{\mathrm{int}}}$ collects the values of $\nabla\!\cdot\!\bigl[(M(u_{\mathrm{NN}}^n)-\Gamma)\,\nabla\mu_{\mathrm{NN}}^{\star,n}\bigr]$ at the interior collocation points. For the GTransNet-BDF2 scheme~\eqref{GBDF2_varmob}, the least-squares system reads
\begin{align}\label{LSQ_varmob}
\begin{pmatrix}
\frac{3}{2}\bm\Phi & -\Gamma\Delta t\,\bm\Phi_{\Delta}\\[2pt]
\varepsilon^2\bm\Phi_{\Delta}-\kappa\bm\Phi & \bm\Phi\\
\bm\Phi_{\mathrm{bd}} & \bm 0\\
\bm 0 & \bm\Phi_{\mathrm{bd}}
\end{pmatrix}\bm c^{n+1}
=\begin{pmatrix}
2\bm\Phi\bm\alpha^n-\tfrac{1}{2}\bm\Phi\bm\alpha^{n-1}+\Delta t\,\tilde{\bm g}_{\mathrm{var}}^{n+1}\\[2pt]
f_\kappa(2\bm\Phi\bm\alpha^n-\bm\Phi\bm\alpha^{n-1})\\
\bm 0\\
\bm 0
\end{pmatrix},
\end{align}
where $\tilde{\bm g}_{\mathrm{var}}^{n+1}\in\mathbb R^{K_{\mathrm{int}}}$ collects the values of $\nabla\!\cdot\!\bigl[(M(\tilde u_{\mathrm{NN}}^{n+1})-\Gamma)\,\nabla\tilde\mu_{\mathrm{NN}}^{n+1}\bigr]$ at the interior collocation points. The systems~\eqref{LSQ_varmob_BDF1}-\eqref{LSQ_varmob} are solved exactly as in Section~\ref{subsec:LSQ}, where the QR factorizations are precomputed once and each time step requires only a matrix-vector product followed by back substitution.

\end{document}